\documentclass[a4paper,11pt]{amsart}
\usepackage{amssymb,amscd,amsxtra}
\usepackage{array}
\usepackage[all]{xy}
\usepackage{enumerate}
\usepackage{color}
\usepackage[pdftex]{hyperref}
\hypersetup{%
bookmarksnumbered=true,%
colorlinks=true,%
setpagesize=false,%
pdftitle={},%
pdfauthor={}}

\theoremstyle{plain}
\newtheorem{Thm}{Theorem}[section]
\newtheorem{Lem}[Thm]{Lemma}

\newtheorem{Prop}[Thm]{Proposition}

\newtheorem{Ques}[Thm]{Question}

\newtheorem*{main}{Main Theorem}

\theoremstyle{definition}
\newtheorem{Def}[Thm]{Definition}
\newtheorem{Def-Lem}[Thm]{Definition-Lemma}
\newtheorem{Cond}[Thm]{Condition}
\newtheorem{Rem}[Thm]{Remark}

\newtheorem*{Ack}{Acknowledgments}

\newcommand{\Aut}{\operatorname{Aut}}
\newcommand{\Bir}{\operatorname{Bir}}

\newcommand{\prt}{\partial}

\newcommand{\Pic}{\operatorname{Pic}}

\newcommand{\HH}{\operatorname{H}}
\newcommand{\I}{\operatorname{I}}
\newcommand{\II}{\operatorname{II}}

\newcommand{\wt}{\operatorname{wt}}

\newcommand{\ord}{\mathrm{ord}}

\newcommand{\mbA}{\mathbb{A}}

\newcommand{\mbC}{\mathbb{C}}

\newcommand{\mbP}{\mathbb{P}}
\newcommand{\mbQ}{\mathbb{Q}}

\newcommand{\mbZ}{\mathbb{Z}}

\newcommand{\mcH}{\mathcal{H}}

\newcommand{\mcO}{\mathcal{O}}

\newcommand{\msp}{\mathsf{p}}
\newcommand{\msq}{\mathsf{q}}

\newcommand{\mbfw}{\mathbf{w}}
\newcommand{\iniw}{\mathbf{w}_{\operatorname{in}}}

\newcommand{\ratmap}{\dashrightarrow}

\makeatletter
\def\imod#1{\allowbreak\mkern10mu({\operator@font mod}\,\,#1)}
\makeatother

\title[Birationally rigid Pfaffian Fano 3-folds]{Birationally rigid Pfaffian Fano 3-folds}
\author[H.~Ahmadinezhad and T.~Okada]{Hamid Ahmadinezhad \and Takuzo Okada}
\address{Department of Mathematical Sciences, Loughborough University, LE11 3TU, UK}
\email{h.ahmadinezhad@lboro.ac.uk}
\address{Department of Mathematics, Faculty of Science and Engineering, Saga University, Saga 845-8502 Japan}
\email{okada@cc.saga-u.ac.jp}
\date{}

\begin{document}

\maketitle

\begin{abstract} We classify birationally rigid orbifold Fano 3-folds of index one defined by $5\times 5$ Pfaffians. We give a sharp criterion for birational rigidity of these families based on the type of singularities that the varieties admit. Various conjectures are born out of our study, highlighting a possible approach to the classification of terminal Fano 3-folds. The birationally rigid cases are the first known rigid examples of Fanos that are not (weighted) complete intersection.
\end{abstract}

%\tableofcontents
\thispagestyle{empty}

\section{Introduction}
A variety $X$ is Fano if its anticanonical class $-K_X$ is ample. They are central in geometry, as any uniruled variety is birational to a Fano or a fibration into Fanos by the Minimal Model Program (MMP).

Smooth Fano 3-folds have been classified by Iskovskikh \cite{Isk1,Isk2} and Mori-Mukai \cite{Mori-Mukai}. 
However, looking at Fano varieties as outputs of MMP, the smoothness condition must be relaxed, and be replaced with $\mbQ$-factorial and terminal. Graded ring approach of Reid provides a list of Fano 3-folds to study. It considers a Fano 3-fold $X$ embedded into a weighted projective space via the anticanonical ring \cite{ABR}
\[R(X,-K_X)=\bigoplus_{n\geq 0}\HH^0(X,-nK_X),\]
and using the numerical datum from such embedding produces families of Fano 3-folds. One approach to the classification of Fano 3-folds would be to study birational relations among these embedded Fanos. However, there are tens of thousands of candidate families, suggesting the impossibility of such study. One of the aims of this article is to convince the reader that it may be enough to consider only a small portion of this list, and hope to eventually get a complete classification. We give evidence that perhaps there are only a few hundreds of families that do not admit Mori fibrations over a curve or a surface. Hence, a full study of relations between those that only admit Fano structures may be possible. Then one goes to study fibration cases and examine their geometry.

\subsection{Birational rigidity of Fanos}
A Fano variety $X$ in the Mori category, that is $\mbQ$-factorial and terminal, is said to be birationally rigid if the only Mori fibre space birational to $X$ is $X$ itself. In other words, $X$ admits no birational structure of a strict Mori fibre space $Y\rightarrow S$ (with $\dim S > 0$) and $X$ is not birational to any other Fano variety. A birationally rigid Fano $X$ is called birationally super-rigid if $\Bir(X)=\Aut(X)$. For example it is known that a smooth hypersurface of degree $n$ in $\mbP^n$ is birationally super-rigid for $n\geq 4$; see \cite{Isk-Manin, pukhlikov, defernex} and \cite{Suzuki} for a generalisation of this. 

The first case of the example above, that is the smooth quartic 3-folds, a celebrated result of Iskovskikh and Manin, was generalised in \cite{CPR} to show that a general quasi-smooth Fano hypersuface of index one in a weighted projective space is birationally rigid. Such Fano $X$ is defined as a hypersurface $\{f=0\}$ of degree $d$ in a weighted projective space $\mbP(a_0,a_1,a_2,a_3,a_4)$, where $\sum a_i-d=1$ (hence the index), the Jacobian of $f$ vanishes only at the origin (hence quasi-smooth), and the singularities on $X$ are inherited from the ambient weighted projective space and are all terminal. There are 95 families with this property. One can consider higher codimension Fanos, for which the number of Fano families are shown in Table~\ref{codim-table}. These number currently only serve as upper bounds, except that in codimensions $1$, $2$ and $3$ they are confirmed to be exact.
\begin{table}[ht]\small%\footnotesize
\[\begin{array}{l||c|c|c|c|c|c|c|c|c|c}
\text{Codimension}&1&2&3&4&5&6&7&\dots&18&\dots\\
\hline
\text{Number of families}&95&85&70&145&164&253&303&\dots&4709&\dots
\end{array}\]
\caption{{\footnotesize Possible number of index one Fano families in each codimension\label{codim-table}}}
\end{table}

As mentioned before, Corti, Pukhlikov and Reid proved that a general member of each family in codimension one is birationally rigid \cite{CPR}. This was generalised by Cheltsov and Park for any quasi-smooth such Fano \cite{cheltsov-park}. The codimension two families were studied by Okada in \cite{OkadaI, OkadaII, OkadaIII}. For instance it was shown that

\begin{Thm}\cite{Isk-Pukh,OkadaI}\label{codim2} Let $X$ be a general quasi-smooth Fano 3-fold of index one embedded in codimension two in a weighted projective space. Then $X$ is birationally rigid if and only if it belongs to one of $18$ specific families.\end{Thm}

Theorem~\ref{codim2}, in particular, generalises a result of Iskovskikh and Pukhlikov that shows a general smooth complete intersection of a conic and a cubic in $\mbP^5$ is birationally rigid, see \cite{Isk-Pukh} and \cite[chapter\,2]{pukh-book}. 

Theorem~\ref{codim2} has been generalised for quasi-smooth models (without the generality conditions) by Ahmadinezhad and Zucconi \cite{Ahm-Zuc}.

It is crucial to note that the birationally rigid cases in Theorem\,\ref{codim2} are those that do not admits a Type~$\I$ centre, which are defined to be:

\begin{Def}[Singularity types] Let $X\subset\mbP=\mbP(a_0,\dots,a_n)$ be a quasi-smooth Fano 3-fold. Suppose the singular point $p\in X$ is a coordinate point of $\mbP$ of local analytic type $\frac{1}{a}(1,b,a-b)$, implying that $n-3$ of the defining polynomials of $X$ are of the form $f_i=x_k^mx_i+\dots$, where $p$ is the $k^\text{th}$ coordinate and $a=a_k$. Suppose the three other weights (the tangent weights) are $a_\alpha,a_\beta$ and $a_\gamma$, then $p$ is of Type $\I$ if $(1, b, a-b)=(a_\alpha,a_\beta,a_\gamma)$, up to reordering, and $K_X^3>\frac{1}{ab(a-b)}$.
%,
%\item $p$ is of Type $\II_1$ if $(1, b, a-b)=(a_\alpha,a_\beta,a_\gamma-a_k)$, up to reordering, and $K_X^3>\frac{1}{ab(a-b)}$.
%\end{enumerate}
These are precisely the images of Type $\I$ unprojections \cite{graded}.

Type $\II_1$ centres are, similarly, the images of Type $\II_1$ unprojections, that is a generic complete intersection Type $\II$ unprojection \cite{Papa1, Papa2}.
\end{Def}

We go further and examine birational rigidity in codimension $3$.

\subsection*{Pfaffian Fanos} A {\it Pfaffian Fano $3$-fold} $X$ is determined by a $5 \times 5$ skew-symmetric matrix $M$, called the {\it syzygy matrix} of $X$, whose entries are homogeneous polynomials in variables $x_0,\dots,x_6$ with suitable weights $\deg x_i = a_i$.
The $3$-fold $X$ is embedded in $\mbP (a_0,\dots,a_6)$ as a codimension $3$ subvariety and it is defined by $5$ Pfaffians $F_1,\dots,F_5$ of $M$.
There are $69$ families of Pfaffian Fano $3$-folds, which form all codimension $3$ Fano $3$-folds of index one together with $X_{2,2,2}\subset\mbP^6$ (the complete intersection of $3$ quadrics in $\mbP^6$). These are studied in details in \cite{ABR}, which represent a success story of the application of Eisenbud-Buchsbaum structure theory of Gorenstein codimenstion 3 ideals \cite{Eis-Buch}. Some explicit examples of these are scattered in this article, see for example Section~\ref{sec:deg42}.

Among these $69$ families only $5$ families do not have a Type $\mathrm{I}$ centre. It was proved by Brown and Zucconi \cite{BZ} that a general Pfaffian Fano with a Type $\I$ centre is birationally non-rigid. The remaining $5$ families are the main objects of this article and the descriptions of syzygy matrix $M$ and defining polynomials $F_1,\dots,F_5$ will be given in the beginnings of Sections \ref{sec:deg42}--\ref{sec:deg4} (see also the table in Section \ref{sec:table}).
Among the above $5$ families, $2$ families have a Type $\mathrm{II}_1$ centre.
The aim of this article is to prove birational (super-)rigidity for the $3$ families which do not admit Type $\mathrm{I}$ or Type $\mathrm{II}_1$ centre and to prove birational non-rigidity of the $2$ families which do not admit a Type $\mathrm{I}$ centre but admit a Type $\mathrm{II}_1$ centre.

\begin{main} Let $X$ be a general Pfaffian Fano 3-fold. Then $X$ is birationally rigid if and only if it does not contain a Type $\I$ or Type $\II_1$ centre.
\end{main}

To summarise, a (general) quasi-smooth Fano in 95 out of 95 families in codimension one, 19 out of 85 families in codimension two and 3 out of 70 families in codimension three are birationally rigid. Consequently, it is very natural to expect an affirmative answer to Question~\ref{rig-question}. Below (Question~\ref{solid-question}) we discuss a more general, and perhaps more fundamental, version of this.

\begin{Ques}\label{rig-question} Does there exist a small $n$, say $n=4$ or $5$, such that for any codimension bigger than $n$ all Fano 3-folds, minimally embedded in a weighted projective space, admit a different Mori fibre space structure, i.e.\ they are all birationally non-rigid?
\end{Ques}

\subsection{Classification of Fano $3$-folds: Solid Fano varieties and Mori fibrations.}

The results of \cite{OkadaII,OkadaIII} go beyond birational rigidity in codimention two and study birigid Fanos in codimension two, following \cite{corti-mella}. Birigid Fanos are Mori fibre space Fanos that are not birationally rigid but birational to only one other Mori fibre space Fano variety. To capture this phenomenon, we introduce the following notion, which we believe will play a central role in the birational classification of Fano $3$-folds.

\begin{Def}\label{solid} A Fano variety is called {\it solid} if it does not admit a birational map to any strict Mori fibre space. By strict Mori fibre space we mean a Mori fibration with positive dimensional base, that is a Mori fibre space with Picard number strictly greater than $1$.
\end{Def}

In particular, \cite{OkadaII} and \cite{OkadaIII} show that $6$ families among the codimension $2$ Fanos are non-solid (birational to del Pezzo fibrations) and the rest are expected to be solid. Following these observations, and based on our experience and our result on the number of rigid Fanos in codimension three, we pose the following question, as  step ahead of Question\,\ref{rig-question}.

\begin{Ques}\label{solid-question} Do solid Fanos exist in higher codimensions? In other words, does there exist a small $n$ such that for any codimension bigger than $n$ all Fano 3-folds admit a structure of a strict Mori fibre space?
\end{Ques}

The evidence, highlighted in this article, suggests that the answer to this question should be ``No''. In that case, it remains to classify solid Fano $3$-folds and consider the non-solid Fanos as the end point of Sarkisov links on del Pezzo fibrations or conic bundles. Then examine birational rigidity of, and birational maps between, del Pezzo fibrations and then similarly for conic bundles; a subject of further study. This will eventually give a {\it hierarchical} classification of Fanos and Mori fibre spaces in dimension three.

\subsection{Notation and Conventions}
We denote by $\msp_{x_i}$ the vertex of $\mbP = \mbP (a_0,\dots,a_6)$ at which only the coordinate $x_i$ does not vanish.
For homogeneous polynomials $G_1,\dots,G_m$, we denote by $(G_1 = \cdots = G_m = 0)$ the closed subscheme of $\mbP$ defined by the homogeneous ideal $(G_1,\dots,G_m)$.
For a polynomial $F$ and a monomial $g$, we write $g \in F$ if the coefficient of $g$ in $F$ is non-zero.
For polynomials $f,g$, we say that $f$ and $g$ are {\it proportional} (denoted $f \sim g$) if there are complex numbers $\lambda,\mu$ with $(\lambda,\mu) \ne (0,0)$ such that $\lambda f - \mu g = 0$.
Let $X$ be a Pfaffian Fano $3$-fold.
We always assume that $X$ is quasi-smooth, that is, its affine cone $C_X = (F_1 = \cdots = F_5= 0) \subset \mbA^7$, where $F_1,\dots,F_5$ are defining polynomials of $X$, is smooth outside the origin.
We set $A = -K_X$.

\begin{Def}
Let $X$ be a Fano $3$-fold.
We say that an extremal divisorial extraction $\varphi \colon Y \to X$ with exceptional divisor $E$ is a {\it maximal extraction} if there is a mobile linear system $\mcH \sim_{\mbQ} - n K_X$, $n \in \mbQ$, such that 
\[
\frac{1}{n} > c (X,\mcH) = \frac{a_E (K_X)}{m_E (\mcH)},
\]
where $c (X,\mcH) = \max \{ \lambda \mid K_X + \lambda \mcH \text{ is canonical} \}$ is the canonical threshold of the pair $(X,\mcH)$, $a_E (K_X)$ is the discrepancy of $K_X$ along $E$ and $m_E (K_X)$ is the multiplicity of $\mcH$ along $E$.
The centre $\varphi(E)$ on $X$ of a maximal extraction is called a {\it maximal centre}.
\end{Def}

\paragraph{\bf The structure of the proof} The proof of birational rigidity will be done by excluding most of the subvarieties as maximal centres and constructing a birational involution centred at the remaining subvarieties. Curves and smooth points are excluded in Section\,\ref{sec:curvenspt}. Section\,\ref{sec:methods} summarises the methods to exclude singular points. Then in each following section we deal with one of the $5$ families, and finally in Section\,\ref{sec:table} we encapsulate the results with a table.

\begin{Ack} We have benefitted from conversation with Gavin Brown, Ivan Cheltsov, Tommaso de Fernex, Miles Reid, Konstantin Shramov and Francesco Zucconi at various occasions. We would like to thank them for their generosity in sharing their knowledge and opinions with us.
The second author is partially supported by JSPS KAKENHI Grant Number 26800019.
\end{Ack}

\section{Exclusion of curves and nonsingular points} \label{sec:curvenspt}

Let $X$ be a Pfaffian Fano $3$-fold.
We first exclude curves as maximal centres.

\begin{Lem} \label{lem:exclcurve}
If $(A^3) \le 1$, then no curve on $X$ is a maximal centre.
\end{Lem}

\begin{proof}
Let $\Gamma \subset X$ be an irreducible and reduced curve.
We may assume that $\Gamma$ is contained in the nonsingular locus of $X$ because otherwise $\Gamma$ passes through a terminal quotient singular point and thus there is no divisorial extraction centred along $\Gamma$ (see \cite{Kawamata}).
By \cite[Lemma 2.9]{OkadaII} (see also \cite[Remark 2.10]{OkadaII} and \cite[Theorem 5.1.1]{CPR}), $\Gamma$ is not a maximal centre if $(A \cdot \Gamma) \ge (A^3)$.
We have $(A \cdot \Gamma) \ge 1$ since $\Gamma$ is contained in the nonsingular locus of $X$.
Thus $\Gamma$ cannot be a maximal centre since $(A \cdot \Gamma) \ge 1 \ge (A^3)$.
\end{proof}

\begin{Prop} \label{prop:exclcurve}
Let $X$ be a Pfaffian Fano $3$-fold without Type $\mathrm{I}$ centre.
Then no curve on $X$ is a maximal centre.
\end{Prop}

\begin{proof}
This follows immediately from Lemma \ref{lem:exclcurve} since $(A^3) \le 1$ in all the cases.
\end{proof}

Next, we exclude nonsingular points as a maximal centre.

\begin{Def}
Let $X$ be a normal projective variety and $\msp \in X$ a nonsingular point.
We say that a Weil divisor class $L$ on $X$ {\it isolates} $\msp$ if $\msp$ is an isolated component of the base locus of the linear system
\[
\mathcal{L}_{\msp}^s := |\mathcal{I}^s_{\msp} (s L)|
\]
for some integer $s > 0$.
\end{Def}

We refer the readers to \cite[Proof of (A) in pages 210 and 211]{CPR} for the proof of the following lemma.
The proof given there is for weighted hypersurfaces but the same argument applies.

\begin{Lem}[\cite{CPR}] \label{lem:criexclnspt}
Let $\msp \in X$ be a nonsingular point of a $\mathbb{Q}$-Fano $3$-fold $X$.
If $l A$ isolates $\msp$ for some $0 < l \le 4/(A^3)$, then $\msp$ is not a maximal centre.
\end{Lem}

Let $\mathbb{P} := \mathbb{P} (a_0,\dots,a_6)$ be the weighted projective $6$-space with homogeneous coordinates $x_0,\dots,x_6$ which is the ambient space of a Pfaffian Fano $3$-fold $X$.
We assume $a_0 \le a_1 \le \cdots \le a_6$.
The following enables us to find an isolating class.

\begin{Lem}[{\cite[Lemma 5.6.4]{CPR}}] \label{lem:isolclass}
Let $\msp \in X$ be a nonsingular point and let $\{g_i\}$ be a finite set of homogeneous polynomials in variables $x_0,\dots,x_6$.
If $\msp$ is a component of the set
\[
X \cap \bigcap (g_i = 0),
\]
then $l A$ isolates $\msp$, where $l = \max \{ \deg g_i\}$.
\end{Lem}

\begin{Lem} \label{lem:nspt}
Suppose that $a_5 a_6 \le 4/(A^3)$.
Then no nonsingular point of $X$ is a maximal centre.
\end{Lem}

\begin{proof}
Let $\msp = (\alpha_0 \!:\! \cdots \!:\! \alpha_6) \in X$ be a nonsingular point.
Then, there exists $k \in \{0,1,\dots,6\}$ such that $\alpha_k \ne 0$.
For $i = 0,1,\dots,6$, we define
\[
m_i = \frac{a_i}{\operatorname{lcm} (a_i,a_k)}.
\]
Then we define $g_i = \alpha_k^{m_i} x_i^{m_k} - \alpha_i^{m_k} x_j^{m_i}$ for $i \ne k$.
We have
\[
X \cap \bigcap_{i \in \{0,1,\dots,6\} \setminus \{k\}} (g_i = 0) = \{ \msp \}.
\]
Moreover, we have
\[
\deg g_i
= \frac{a_i a_k}{\operatorname{lcm} (a_i,a_k)} \le a_5 a_6
\]
for any $i \ne k$.
It follows from Lemma \ref{lem:isolclass} that $l A$ isolates $\msp$ for some $l \le a_5 a_6$.
Now the assumption $a_5 a_6 \le 4/(A^3)$ and Lemma \ref{lem:criexclnspt} complete the proof.
\end{proof}

\begin{Prop} \label{prop:exclnspt}
Let $X$ be a Pfaffian Fano $3$-fold without Type $\mathrm{I}$ centre.
Then no nonsingular point on $X$ is a maximal centre.
\end{Prop}

\begin{proof}
The condition $a_5 a_6 \le 4/(A^3)$ is satisfied for Pfaffian Fano $3$-folds $X$ of degree $1/42, 1/30,1/20$ and $1/12$.
Thus the assertion for these $4$ families follows from Lemma \ref{lem:nspt}.

It remains to consider a Pfaffian Fano $3$-fold $X$ of degree $1/4$.
Let $x,y,z_0,z_1,t_0,t_1,u$ be the homogeneous coordinates of the ambient space $\mbP (1,2,3^2,4^2,5)$ and $\msp \in X$ a nonsingular point.
Let $\pi \colon X \to \mathbb{P} := \mathbb{P} (1,2,3^2,4^2)$ be the projection from $\msp_u$ which is indeed a morphism since $\msp_u \notin X$ (see the table in Section \ref{sec:table}).
Since there are monomials $x^{12}$, $y^6$, $z_0^4$, $z_1^4$, $t_0^3$ and $t_1^3$ of degree $12$, we can find homogeneous polynomials $g_1,\dots,g_m$ as suitable linear combinations of those monomials such that 
\[
\bigcap (g_i = 0) = \{ \pi (\msp) \}
\] 
on $\mathbb{P}$.
It follows that we have
\[
X \cap \bigcap (g_i = 0) = \pi^{-1} (\pi (\msp)),
\]
and the right-hand side consists of finitely many points including $\msp$ since $\pi$ does not contract a curve.
This shows that $12 A$ isolates $\msp$, hence $\msp$ cannot be a maximal centre since $12 < 4/(A^3) = 16$.
This completes the proof.
\end{proof}

\section{Excluding methods for singular points}\label{sec:methods}

We will exclude singular points as a maximal centre (or construct a Sarkisov link) on Pfaffian Fano $3$-folds without Type $\mathrm{I}$ center in the subsequent sections.
In this section we explain the methods excluding singular points.

We fix some notation which will be valid in the rest of this paper.
Let $X$ be a Pfaffian Fano $3$-fols and $\msp \in X$ a singular point.
Let $\msp$ be of type $\frac{1}{r} (1,a,r-a)$.
We denote by $\varphi \colon Y \to X$ the Kawamata blowup of $X$ at $\msp$, that is, the weighted blowup with weight $\frac{1}{r} (1,a,r-a)$.
Note that $\varphi$ is the unique extremal divisorial extraction centred at the terminal quotient singular point $\msp$ (see \cite{Kawamata}).
We denote by $E$ the exceptional divisor of $\varphi$.
We set $A = -K_X$ and $B = -K_Y = \varphi^* A - \frac{1}{r} E$.
We will frequently compute intersection numbers of divisors on $Y$ and this is done by the formula
\[
(\varphi^*A^2 \cdot E) = (\varphi^*A \cdot E^2) = 0, \ 
(E^3) = \frac{r^2}{a (r-a)}.
\]
For a curve or a divisor $\Delta \subset X$, we denote by $\tilde{\Delta}$ its proper transform $\varphi_*^{-1} \Delta$ via $\varphi$.
We will exclude singular points on $X$ by applying the following criteria.

\begin{Lem}[{\cite[Corollary 2.17]{OkadaII}}] \label{lem:excltc}
If $(L \cdot B^2) \le 0$ for some nef divisor $L$ on $Y$, then $\msp \in X$ is not a maximal centre.
\end{Lem}

\begin{Lem}[{\cite[Lemma 2.18]{OkadaII}}] \label{lem:exclbadC}
Assume that there are surfaces $S$ and $T$ on $Y$ with the following properties.
\begin{enumerate}
\item $S \sim_{\mbQ} a B + d E$ and $T \sim_{\mbQ} b B + e E$ for some integers $a,b,d,e$ such that $a, b > 0$, $0 \le e \le a_E (K_X) b$ and $a e - b d \ge 0$.
\item The intersection $\Gamma := S \cap T$ is a $1$-cycle whose support consists of irreducible and reduced curves which are numerically proportional to each other.
\item $(T \cdot \Gamma) \le 0$.
\end{enumerate}
Then, $\msp \in X$ is not a maximal extraction.
\end{Lem}

Note that in Lemma \ref{lem:exclbadC}, the condition (3) is equivalent to the condition $(T \cdot S \cdot T) \le 0$.

When we apply Lemma \ref{lem:excltc}, we need to find a nef divisor on $Y$, which will be done by the following result.

\begin{Lem}[{\cite[Lemma 6.6]{OkadaII}}] \label{lem:isolnef}
Suppose that there are prime divisors $D_1,\dots,D_k$ on $X$ with the following properties.
\begin{enumerate}
\item The intersection $D_1 \cap \cdots \cap D_k$ does not contain a curve passing through $\msp$.
\item For $i = 1,2,\dots,k$, $\tilde{D}_i$ is $\mbQ$-linearly equivalent to $b_i B + e_i E$ for some $b_i > 0$ and $e_i \ge 0$.
\item We have $c \le a_E (K_X)$, where $c = \max \{e_i/b_i\}$ and $a_E (K_X)$ is the discrepancy of $K_X$ along $E$.
\end{enumerate}
Then, the divisor $L = B + c E$ is nef.
\end{Lem}

\begin{Def}
Let $X$ be a Pfaffian Fano $3$-fold and $\msp \in X$ a (singular) point. 
We say that $\{f_1,\dots,f_k\}$, where $f_1,\dots,f_k$ are homogeneous polynomials, {\it isolates} $\msp$ if $(f_1 = \cdots = f_k = 0) \cap X$ does not contain a curve passing through $\msp$.
\end{Def}

Suppose that $\{f_1,\dots,f_k\}$ isolates a singular point $\msp \in X$ and let $D_i = (f_i = 0) \cap X$. 
Then $D_1,\dots,D_k$ satisfy (1) of Lemma \ref{lem:isolnef}.
We see 
\[
\tilde{D}_i = b_i \varphi^*A - \ord_E (f_i) E = b_i B + \frac{b_i - r \ord_E (f_i)}{r} E,
\] 
where $b_i = \deg f_i$ and $r$ is the index of the singularity $\msp \in X$.
It follows from Lemma \ref{lem:isolnef} that $L = B + c E$ is nef on $Y$, where 
\[
c = \max \left\{ \frac{b_i - r \ord_E (f_i)}{b_i r} \right\}
\]
if $b_i \ge r \ord_E (f_i)$ for every $i$ and $c \le \frac{1}{r}$.

In the course of excluding singular points or constructing Sarkisov links, it is necessary to understand geometric objects on $Y$ (e.g.\ proper transforms of curves or divisors on $X$ and their intersections).
We will explain explicit descriptions of Kawamata blowups $\varphi \colon Y \to X$ in terms of the embedded weighted blowup of $X \subset \mbP$ at $\msp$ in a general setting.

From now on until the end of this section, we work in a more general setting.
Let $X$ be a normal projective $\mbQ$-factorial $3$-fold defined by homogeneous polynomials $F_1, \cdots, F_m \in \mbC [x_0,\dots,x_{n+3}]$ in a weighted projective space $\mbP = \mbP (a_0,\dots,a_{n+3})$ with homogeneous coordinates $x_0,\dots,x_{n+3}$ and $\msp$ a terminal quotient singular point of type $\frac{1}{r} (1,a,r-a)$ and $\varphi \colon Y \to X$ the Kawamata blowup of $X$ at $\msp$ with exceptional divisor $E$.
%Here $x_0,\dots,x_{n+3}$ be the homogeneous coordinates of $\mbP$.
We explain the computation of the vanishing order of a section along $E$ in the case where $\msp$ can be transformed into a vertex by a coordinate change.

\begin{Def}
Let $\psi \colon V \to X$ be a birational morphism from a normal projective variety $V$ and $F$ an prime excetional divisor of $\psi$.
For a global section $s \in H^0 (X, \mcO_X (d))$,  we denote by $\ord_F (s)$ the rational number such that $\psi^* (s = 0) = \psi_*^{-1} (s = 0) + \ord_F (s) F$ and call it the {\it vanishing order} of $s$ along $F$.
For global sections, $s_1,\dots,s_m$, the expressions $\ord_E (s_1,\dots,s_m) = \frac{1}{r} (b_1,\dots,b_m)$ and $\ord_E (s_1,\dots,s_m) \ge \frac{1}{r} (b_1,\dots,b_m)$ mean $\ord_E (s_i) = \frac{b_i}{r}$ and $\ord_E (s_i) \ge \frac{b_i}{r}$ respectively for $i = 1,\dots,m$.
\end{Def}

%Let $F_1,\dots,F_m \in \mbC [x_0,\dots,x_n]$ be defining polynomials of $X$.
We assume $\msp = \msp_{x_0}$.
In this case $r = a_0$.
Then $X$ is quasi-smooth at $\msp$ if and only if, after re-ordering $x_1,\dots,x_{n+3}$ and $F_1,\dots,F_m$, we have $x_0^{l_1} x_1 \in F_1, \dots, x_0^{l_{n-3}} x_n \in F_n$ for some $l_1,\dots,l_n > 0$.
In this case, we have $a_{n+1} \equiv 1, a_{n+2} \equiv a, a_{n+3} \equiv r-a \pmod{r}$, after re-oredering $x_{n+1},x_{n+2},x_{n+3}$, and the Kawamata blowup $\varphi \colon Y \to X$ is the weighted blowup with weight $\wt (x_{n+1},x_{n+2},x_{n+3}) = \frac{1}{r} (1,a,r-a)$.

We work on the open subset $U$ of $X$ where $x_0 \ne 0$.
For a polynomial $G (x_0,x_1,\dots,x_{n+3})$, we denote $G|_{x_0 = 1} = G (1,x_1,\dots,x_{n+3})$.
Then $U$ is the geometric quotient of the affine scheme 
\[
V = (F_1|_{x_0 = 1} = \cdots = F_m|_{x_0 = 1} = 0) \subset \mbA^{n+3}
\] 
by the $\mbZ_r$-action given by $x_i \mapsto \zeta^{a_i} x_i$, where $\zeta$ is a primitive $r$th root of unity. 
We see that the defining polynomials $F_{n+1}, \dots, F_m$ are redundant around $\msp$ since $V$ is a local complete intersection (nonsingular) at its origin (whose image on $U$ is the point $\msp$).

\begin{Def}
For a positive integer $a$, we denote by $\bar{a}$ the positive integer such that $\bar{a} \equiv a \pmod{r}$ and $0 < \bar{a} \le r$.

We say that
\[
\mbfw = \frac{1}{r} (b_1,b_2,\dots,b_{n+3})
\]
is an {\it admissible weight} with respect to $(X,\msp)$ if $b_1,\dots,b_6$ are positive integer such that $b_i \equiv a_i \pmod{r}$ for $i = 1,\dots,n+3$.
We call
\[
\iniw := \frac{1}{r} (\bar{a}_1,\bar{a}_2,\dots,\bar{a}_{n+3})
\]
the {\it initial weight} of $(X, \msp)$.
\end{Def}

Note that $\bar{r} = r$ by the above definition.
Note also that the initial weight is admissible.
For an admissible weight $\mbfw$, we can associate the weighted blowup $\Phi_{\mbfw} \colon Q_{\mbfw} \to \mbP$ at $\msp$ with $\wt (x_1,\dots,x_{n+3}) = \mbfw$.
We see that the exceptional divisor of $\Phi_{\mbfw}$ is isomorphic to the weighted projective space $\mbP (b_1,b_2,\dots,b_{n+3})$ with coordinates $x_1,\dots,x_{n+3}$.
Here, by a slight abuse of notation, we use $x_i$ for the coordinates of $\mbP (b_1,\dots,b_{n+3})$.
In this case, $x_i$ has weight $b_i$ and this $x_i$ is different from the $x_i$ of $\mbP$. 
We denote by $Y_{\mbfw}$ the proper transform of $X$ via $\Phi_{\mbfw}$, by $\varphi_{\mbfw} \colon Y_{\mbfw} \to X$ the induced birational morphism and by $E_{\mbfw}$ the exceptional divisor of $\varphi_{\mbfw}$.

\begin{Def}
Let $\mbfw$ be an admissible weight.
For $i = 1,\dots,n$, we denote by $F_i^{\mbfw}$ the lowest weight part of $F_i|_{x_0=1}$ with respect to the $\mbfw$-weight.
We say that $\mbfw$ satisfies the {\it Kawamata blowup condition} (abbreviated as {\it KBL condition}) if $x_i \in F_i^{\mbfw}$ for any $i = 1,\dots,n$ and $b_i = \bar{a}_i$ for $i = n+1,n+2,n+3$ (i.e. $(b_{n+1},b_{n+2},b_{n+3}) = (1,a,r-a)$).
\end{Def}

Suppose that $\mbfw$ is an admissible weight which satisfies the KBL condition.
Then we have an isomorphism
\[
E_{\mbfw} \cong (F_1^{\mbfw} = F_2^{\mbfw} = \cdots = F_n^{\mbfw} = 0) \subset \mbP (b_1,\dots,b_{n+3}).
\]
Since $x_i \in F_i^{\mbfw}$ for $i = 1,\dots,n$ and $b_{n+1} = 1$, $b_{n+2} = a$, $b_{n+3} = r-a$, we have an isomorphism $E_{\mbfw} \cong \mbP (1,a,r-a)$ by eliminating $x_1,\dots,x_n$.
Moreover $\varphi_{\mbfw}$ is the Kawamata blowup of $X$ at $\msp$ (see Remark \ref{rem:embKblup}).

\begin{Rem} \label{rem:embKblup}
Let $\mbfw = \frac{1}{r} (b_1,\dots,b_{n+3})$ be an admissible weight satisfying KBL condition.
We explain that $\varphi_{\mbfw} \colon Y_{\mbfw} \to X$ is indeed the Kawamata blowup at $\msp$.

The congruence condition $b_i \equiv a_i \pmod{r}$ ensures that the embedded weighted blowup of $U \subset \mbA^{n+3}$ at the origin with weight $\wt (x_1,\dots,x_{n+3}) = (b_1,\dots,b_{n+3})$ is compatible with the $\mbZ_r$-action on $U \subset \mbA^{n+1}$ and gives a well-defined embedded weighted blowup of $X \subset \mbP$ at $\msp$, which is $\varphi_{\mbfw} \colon Y_{\mbfw} \to X$.
As explained above, the $\varphi_{\mbfw}$-exceptional divisor $E_{\mbfw}$ is isomorphic to $\mbP (1,a,r-a)$.
The singular locus of $Y_{\mbfw}$ along $E_{\mbfw}$ is contained in the singular locus of $E_{\mbfw}$.
Let $\msp_a$ and $\msp_{r-a}$ be the points of $E_{\mbfw}$ which corresponds to the points $(0\!:\!1\!:\!0)$ and $(0\!:\!0\!:\!1)$ of $\mbP (1,a,r-a)$, respectively.
Note that $E_{\mbfw}$ is nonsingular outside $\{\msp_a, \msp_{r-a}\}$, and $\msp_a$ (resp.\ $\msp_{r-a}$) is a singular point of $E_{\mbfw}$ if and only if $a > 1$ (resp.\ $r-a > 1$).
In view of the KBL condition, it is straightforward to check that the singularity of $Y_{\mbfw}$ at $\msp_a$ (resp.\ $\msp_{r-a}$) is of type $\frac{1}{a} (1,r-a,-1)$ (resp.\ $\frac{1}{r-a} (1,a,-1)$) when $a > 1$ (resp.\ $r-a > 1$).
This shows that $\varphi_{\mbfw}$ is an extremal divisorial contraction centered at the terminal quotient singular point $\msp$.
By the uniqueness of such a divisorial contraction (\cite{Kawamata}), we conclude that $\varphi_{\mbfw}$ is indeed the Kawamata blowup at $\msp$.
\end{Rem}

From now on, we explain the computation of $\ord_E (x_i)$.
It is clear that $\ord_E (x_{n+1},x_{n+2},x_{n+3}) = \frac{1}{r} (1,a,r-a)$.

\begin{Lem} \label{lem:ordE}
Let $\mbfw$ be an admissible weight satisfying the KBL condition.
Then the following hold.
\begin{enumerate}
\item $\ord_E (x_i) \ge b_i/r$ for $i = 1,\dots,n$.
\item If $F_i^{\mbfw}$ consists only of $x_i$ for some $i = 1,\dots,n$, then $\ord_E (x_i) \ge (b_i + r)/r$.
\item If $F_i^{\mbfw}$ consists only of $x_i$ for some $i = 1,\dots,n$, then the weight
\[
\mbfw' = \frac{1}{r} (b'_1,\dots,b'_n,1,a,r-a),
\]
where $b'_j = b_j$ for $j \ne i$ and $b'_i = b_i + r$, satisfies the KBL condition.
\end{enumerate}
\end{Lem} 

\begin{proof}
We see that $\varphi_{\mbfw}$ is the Kawamata blowup of $X$ at $\msp$ since $\mbfw$ satisfies the KBL condition.
It is clear that $x_i$ vanishes along $E_{\mbfw}$ to order at least $b_i/r$ so that we have $\ord_E (x_i) = \ord_{E_{\mbfw}} (x_i) \ge b_i/r$.
This shows (1).

We prove (3).
We have $x_j \in F_j^{\mbfw}$ for $j = 1,\dots,n$ since $\mbfw$ satisfies the KBL condition.
For a monomial $g$ in variables $x_1,\dots,x_{n+3}$, the $\mbfw'$-weight of $g$ is greater than or equal to the $\mbfw$-weight.
This implies that if there is a monomial $g \in F_j^{\mbfw}$ whose $\mbfw'$-weight and $\mbfw$-weight are the same, then $g \in F_j^{\mbfw'}$.
If $j \ne i$, then the $\mbfw$-weight and $\mbfw'$-weight of $x_j$ coincide so that $x_j \in F^{\mbfw'}_j$.
We have $F^{\mbfw}_i = \alpha x_i$ for some $\alpha \in \mbC \setminus \{0\}$ and any other monomials in $F_i|_{x_0=1}$ has $\mbfw$-weight at least $(b_i+r)/r$.
Hence any monomial in $F_1|_{x_0=1}$ other than $x_i$ has $\mbfw'$-weight at least $(b_i+r)/r$.
Since the $\mbfw'$-weight of $x_i$ is $(b_i+r)/r$, we see $x_i \in F^{\mbfw'}_i$.
This proves (3).
Finally, (2) follows from (1) and (3).
\end{proof}

As an immediate consequence, we have the following somewhat obvious fact: $\ord_E (x_i) \ge \bar{a}_i/r$ for any $1 \le i \le n+3$.

In most of the case, if $x_i$ is chosen as a general member of $H^0 (X, \mcO_X (a_i))$, then we have $\ord_E (x_i) = \frac{\bar{a}_i}{r}$.
Sometimes we seek for a coordinate $x_i$ with high vanishing order and we explain how to obtain such a coordinate.
In general the lowest weight part $F_i^{\iniw}$ with respect to the initial weight $\iniw$ contains a monomial  other than $x_i$.
Now we suppose that, after replacing $x_1$ suitably, the terms in $F_1^{\iniw}$ other than $x_1$ can be eliminated, that is, $F_1^{\iniw} = x_1$.
Then, by Lemma \ref{lem:ordE}, we have $\ord_E (x_1) \ge \frac{\bar{a}_1 + r}{r}$.
We can possibly repeat this process for some coordinates $x_i$ with $i = 1,2,3$ by replacing $\iniw$ with $\mbfw = \frac{1}{r} (\bar{a}_1 + r, \bar{a}_2,\dots,\bar{a}_{n+3})$, which satisfies KBL condition by Lemma \ref{lem:ordE}, and we can obtain coordinates $x_i$ which vanish along $E$ to an order high than $\bar{a}_i/r$.

We will frequently apply the following simple coordinate change technique.

\begin{Lem} \label{lem:elim}
Let $F$ be a polynomial of the form 
\[
F = x_0^3 f_1 + x_0^2 (\alpha x_1 + f_2) + x_0 (x_1 f_3 + f_4) + x_1^2 f_5 + x_1 f_6 + f_7,
\]
where $\alpha \in \mbC \setminus \{0\}$ and $f_i \in \mbC [x_2,\dots,x_n]$.
Then, after replacing $x_1$ with $\gamma x_1 + h$ for suitable $\gamma \in \mbC \setminus \{0\}$ and $h \in \mbC [x_0,x_2,\dots,x_n]$, the terms divisible by $x_0^2$ in $F$ except for $\alpha x_0^2 x_1$ are eliminated.
\end{Lem}

\begin{proof}
We may assume $\alpha = 1$.
Then the replacement $x_1 \mapsto x_1 - y f_1 - f_2 + f_1 f_3 - f_1^2 f_5$ eliminates the terms divisible by $x_0^2$ except for $x_0^2 x_1$.
\end{proof}

\section{Pfaffian Fano $3$-fold of degree $1/42$} \label{sec:deg42}

Let $X = X_{16,17,18,19,20} \subset \mbP (1_x,5_y,6_z,7_t,8_u,9_v,10_w)$ be a Pfaffian Fano $3$-fold of degree $1/42$.
Here a degree of a Fano threefold means the anticanonical degree so that $(A^3) = 1/42$, where $A = -K_X$.
We exclude all the singular points on $X$ and prove that $X$ is birationally super-rigid under a suitable generality condition.
The syzygy matrix of $X$ and the defining polynomials are given as follows:
\[
M =
\begin{pmatrix}
0 & a_6 & a_7 & a_8 & a_9 \\
& 0 & b_8 & b_9 & b_{10} \\
& & 0 & c_{10} & c_{11} \\
& & & 0 & d_{12} \\
& & & & 0
\end{pmatrix}
\hspace{1.5cm}
\begin{aligned}
F_1 &= a_6 c_{10} - a_7 b_9 + a_8 b_8 \\ 
F_2 & = a_6 c_{11} - a_7 b_{10} + a_9 b_8 \\ 
F_3 &= a_6 d_{12} - a_8 b_{10} + a_9 b_9 \\ 
F_4 &= a_7 d_{12} - a_8 c_{11} + a_9 c_{10} \\
F_5 &= b_8 d_{12} - b_9 c_{11} + b_{10} c_{10}
\end{aligned}
\]
Here the entries $a_i,b_i,c_i,d_i$ of $M$ are homogeneous polynomials of (weighted) degree $i$.
The basket of singularities of $X$, which indicates the number and type of singularities, is as follows
\[
\left\{ \frac{1}{2} (1,1,1), \frac{1}{3} (1,1,2), \frac{1}{5} (1,1,4), \frac{1}{5} (1,2,3), \frac{1}{7} (1,1,6) \right\}.
\]

The aim of this section is to prove the following theorem, which will follow from Propositions \ref{prop:exclcurve}, \ref{prop:exclnspt} and the results of the present section (see also \cite[Theorem 2.32]{OkadaII}).
The condition in the statement will be introduced later.

\begin{Thm}
Let $X$ be a Pfaffian Fano $3$-fold of degree $1/42$.
If $X$ satisfies \emph{Condition \ref{cd:deg42-5}}, then it is birationally super-rigid.
\end{Thm}

\subsection{Exclusion of the $\frac{1}{2} (1,1,1)$ point}

\begin{Lem} \label{lem:deg42excl2}
The point of type $\frac{1}{2} (1,1,1)$ is not a maximal centre.
\end{Lem}

\begin{proof}
Let $\msp$ be the point of type $\frac{1}{2} (1,1,1)$.
It is clear that the set $\{x,y,t,v\}$ isolates the point $\msp$ and $\ord_E (x,y,t,v) \ge \frac{1}{2} (1,1,1,1)$.
Thus, we see that $L = 9 \varphi^*A - \frac{1}{2} E$ is nef by Lemma \ref{lem:isolnef} and we compute
\[
(L \cdot B^2) = 9 (A^3) - \frac{1}{2^3} (E^3) = \frac{9}{42} - \frac{1}{2} < 0.
\]
Therefore, $\msp$ is not a maximal centre by Lemma \ref{lem:excltc}.
\end{proof}

\subsection{Exclusion of the $\frac{1}{3} (1,1,2)$ point}

\begin{Lem}
The point of type $\frac{1}{3} (1,1,2)$ is not a maximal centre.
\end{Lem}

\begin{proof}
Let $\msp$ be the point of type $\frac{1}{3} (1,1,2)$.
We set $\Pi = (x = y = t = u = 0)$.
Then $F_5|_{\Pi} = \alpha w^2$ with $\alpha \ne 0$ since $X$ does not contain $\msp_w$.
It follows that 
\[
\Pi \cap X = (x = y = t = u = w = 0) \cap X = \{\msp\}
\]
and $\{x,y,t,u\}$ isolates $\msp$.
We see $\ord_E (x,y,t,u) \ge \frac{1}{3} (1,2,1,2)$.
It follows that $L = 7 \varphi^*A - \frac{1}{3} E$ is nef by Lemma \ref{lem:isolnef} and we compute
\[
(L \cdot B^2) = 7 (A^3) - \frac{1}{3^3} (E^3) = \frac{7}{42} - \frac{1}{6} = 0.
\]
Therefore, $\msp$ is not a maximal centre by Lemma \ref{lem:excltc}.
\end{proof}

\subsection{Exclusion of the $\frac{1}{7} (1,1,6)$ point}

\begin{Lem} \label{lem:deg42excl3}
The point of type $\frac{1}{7} (1,1,6)$ is not a maximal centre.
\end{Lem}

\begin{proof}
We claim that $\{x,y,z\}$ isolates the point $\msp = \msp_t$ of type $\frac{1}{7} (1,1,6)$.
Set $\Pi = (x = y = z = 0)$.
Then we have 
\[
F_1|_{\Pi} = \alpha v t + \beta u^2, \ 
F_3|_{\Pi} = \gamma w u + \delta v^2, \ 
F_5|_{\Pi} = \varepsilon w^2,
\]
for some $\alpha,\beta,\dots,\varepsilon \in \mbC$.
We see that none of $\beta, \delta,\varepsilon$ is zero since $\msp_w, \msp_v, \msp_u \notin X$.
It follows that
\[
\Pi \cap X \subset (x = y = z = u = v = w = 0) = \{ \msp\},
\]
that is, $\{x,y,z\}$ isolates $\msp$.
We see $\ord_E (x,y,z) \ge \frac{1}{7} (1,5,6)$ so that $L = \varphi^*A - \frac{1}{7} E$ is nef by Lemma \ref{lem:isolnef}.
We compute
\[
(L \cdot B^2) = (B^3) = (A^3) - \frac{1}{7^3} (E^3) = \frac{1}{42} - \frac{1}{42} = 0.
\]
Therefore, $\msp$ is not a maximal centre by Lemma \ref{lem:excltc}.
\end{proof}

\subsection{Exclusion of the $\frac{1}{5} (1,1,4)$ point}

Let $\msp$ be the point of type $\frac{1}{5} (1,1,4)$.
After replacing coordinates, we assume $\msp = \msp_z$.
We see $u^2 \in F_1$, $z^3, v^2 \in F_3$ and $w^2 \in F_5$ since $\msp_z,\msp_u,\msp_v,\msp_w \notin X$, and this implies $z \in a_6$, $z^2 \in d_{12}$, $u \in a_8, b_8$, $v \in a_9,b_9$ and $w \in b_{10}, c_{10}$.
We claim $t \in a_7$.
Indeed, if $t \notin a_7$, then $t w \notin F_2$ and this implies that $X$ is not quasi-smooth at the $\frac{1}{7} (1,1,6)$ point $\msp_t$.
This shows $t \in a_7$.
Moreover, since $\msp$ is of type $\frac{1}{5} (1,1,4)$, we have $y^2 z \notin F_1$, which implies $y^2 \notin c_{10}$.
By quasi-smoothness of $X$ at $\msp$, we have $y^2 u \in F_3$, which implies $y^2 \in b_{10}$.
By setting $\Pi = (x = w = 0)$ and by re-scaling coordinates, the restrictions of the syzygy matrix and defining polynomials to $\Pi$ can be written as follows:
\[
M|_{\Pi} =
\begin{pmatrix}
0 & z & t & \alpha u & \beta v \\
& 0 & u & v & y^2 \\
& & 0 & 0 & \gamma z y \\
& & & 0 & \delta z^2 + \varepsilon t y \\
& & & & 0
\end{pmatrix}
\hspace{1.5cm}
\begin{aligned}
F_1|_{\Pi} &= - t v + \alpha u^2 \\ 
F_2|_{\Pi} &= \gamma z^2 y - t y^2 + \beta v u \\
F_3|_{\Pi} &= \delta z^3 + \varepsilon t z y - \alpha u y^2 + \beta v^2 \\
F_4|_{\Pi} &= \delta t z^2 + \varepsilon t^2 y - \alpha \gamma u z y \\
F_5|_{\Pi} &= \delta u z^2 + \varepsilon u t y - \gamma v z y
\end{aligned}
\]
where $\alpha,\beta,\delta \in \mbC \setminus \{0\}$ and $\gamma, \varepsilon \in \mbC$.
By quasi-smoothness of $X$ at the $\frac{1}{7} (1,1,6)$ point $\msp_t$, we have $t^2 y \in F_4$, which implies $\varepsilon \ne 0$.
We set $S = (x = 0) \cap X$ and $T = (w = 0) \cap X$.
Then $\Gamma := S \cap T$ is defined by the equations $F_1|_{\Pi} = \cdots = F_5|_{\Pi} = 0$.
We see $\ord_E (x,z,t,u,v,w) \ge \iniw = \frac{1}{5} (1, 1, 2, 3, 4, 5)$.
Note that $y^3 x^2, y^2 t, y^2 z x$ and $y^2 z^2$ are the monomials of degree $17$ whose initial weight is $2/5$.
The coefficients of $t y^2$ and $z^2 y$ in $F_2$ are $-1$ and $\gamma$, respectively, and let $\lambda$, $\mu$ be the coefficients of $y^3 x^2$, $y^2 z x$ in $F_2$, respectively.
We define $g = - t y + \gamma z^2 + \lambda y^2 x^2 + \mu y z x$.
Then we can write $F_2 = y g + G$, where each monomial in $G$ vanishes along $E$ to order at least $7/5$, hence $\ord_E (g) \ge 7/5$. 
We set $s = g|_{\Pi} = - t y + \gamma z^2$, so that we have  $F_2|_{\Pi} = y s + \beta u v$.

\begin{Cond} \label{cd:deg42-5}
Under the above choice of coordinates, $\gamma \ne 0$ and $\delta + \gamma \varepsilon \ne 0$.
\end{Cond}

\begin{Lem}
If $X$ satisfies \emph{Condition \ref{cd:deg42-5}}, then $\msp$ is not a maximal centre.
\end{Lem}

\begin{proof}
We will show that $\{x,w,g\}$ isolates $\msp$, or equivalently $\{x,w,s\}$ isolates $\msp$.
We set $\Sigma = (x = w = s = 0) \cap X = X \cap \Pi \cap (s = 0)$.
We see $v u = 0$ on $\Sigma$ since $F_2|_{\Pi} = y s + \beta v u$ and $\beta \ne 0$.
By the equation $F_1|_{\Pi} = 0$ and $\alpha \ne 0$, $v = 0$ implies $u = 0$, hence   
\[
\Sigma = (x = w = s = u = t v = \delta z^3 + \varepsilon t z y + \beta v^2 = \delta t z^2 + \varepsilon t^2 y = \gamma v z y = 0),
\]
set-theoretically.
By the assumption $\delta + \gamma \varepsilon \ne 0$, $s = -t y + \gamma z^2$ is not proportional to $\delta z^2 + \varepsilon t y$, so that $(s = \delta z^2 + \varepsilon t y = 0) = (z = t y =0)$.
Hence, it is straightforward to see $\Sigma = \{\msp_y,\msp_t\}$, which shows that $\{x,w,g\}$ isolates $\msp$.

We have $\ord_E (x,w,s) \ge \frac{1}{5} (1,5,7)$ so that $L = 10 \varphi^*A - \frac{5}{5} E$ is nef by Lemma \ref{lem:isolnef} and we compute
\[
(L \cdot B^2) = 10 (A^3)  - \frac{5}{5^3} (E^3) = \frac{10}{42} - \frac{1}{4} < 0.
\]
Therefore, $\msp$ is not a maximal centre by Lemma \ref{lem:excltc}. 
\end{proof}

\subsection{Exclusion of the $\frac{1}{5} (1,2,3)$ point}

Let $\msp$ be the point of type $\frac{1}{5} (1,2,3)$.
We may assume $\msp = \msp_y$.
By the same argument as in the previous subsection, we have $t \in a_7$, $u \in a_8,b_8$, $v \in a_9,b_9$ and $w \in b_{10}, c_{10}$.
Since $\msp$ is of type $\frac{1}{5} (1,2,3)$, we have $v y^2 \in F_4, w y^2 \in F_5$ and $t y^2 \notin F_2, u y^2 \notin F_3$.
We see that $v y^2 \in F_4$ implies $y^2 \in c_{10}$ and $t y^2 \notin F_2$ implies $y^2 \notin b_{10}$.
Since $\msp_t \in X$ is of type $\frac{1}{7} (1,1,6)$, we have $t^2 y \in F_4$, which implies $t y \in d_{12}$.
Moreover, we have $z^3 \in F_3$ since $\msp_z \notin X$, which implies $z \in a_6$.
Hence $y^2 z \in F_1$.
By Lemma \ref{lem:elim}, we can assume that $y^2 z$ is the unique monomial in $F_1$ is divisible by $y^2$ after replacing $z$.

We set $S = (x = 0) \cap X$, $T = (z = 0) \cap X$, $\Gamma = S \cap T$ and $\Pi = (x =z = 0)$.
Then, the restrictions of the syzygy matrix and the defining polynomials to $\Pi$ can be written as follows
\[
M|_{\Pi} =
\begin{pmatrix}
0 & 0 & t & \alpha u & \beta v \\
& 0 & u & v & w \\
& & 0 & \gamma w + \delta y^2 & 0 \\
& & & 0 & t y \\
& & & & 0
\end{pmatrix}
\hspace{1.5cm}
\begin{aligned}
F_1|_{\Pi} &= - t v + \alpha u^2 \\
F_2|_{\Pi} &= - t w + \beta u v \\
F_3|_{\Pi} &= - \alpha u w + \beta v^2 \\
F_4|_{\Pi} &= t^2 y + \beta \gamma w v + \beta \delta v y^2 \\
F_5|_{\Pi} &= u t y + \gamma w^2 + \delta w y^2.
\end{aligned}
\]

Note that $\Gamma$ is defined in $\Pi$ by the above $5$ polynomials.
Note also that none of $\alpha,\beta,\gamma$ and $\delta$ is zero.

\begin{Lem}
$\Gamma$ is an irreducible and reduced curve.
\end{Lem}

\begin{proof}
By setting $t = 1$, we work on the open subset $U \subset X$ on which $t \ne 0$.
By the equations $F_1|_{\Pi} = F_2|_{\Pi} = 0$, we can eliminate $v = \alpha u^2$ and $w = \beta uv = \alpha \beta u^3$.
Hence $\Gamma \cap U$ is isomorphic to the quotient of
\[
(y + \alpha^2 \beta^2 \gamma u^5 + \alpha \beta \delta u^2 y^2 = 0) \subset \mbA^2_{y,u}
\]
under the natural $\mbZ_7$-action.
Thus $\Gamma \cap U$ is an irreducible and reduced affine curve.
We have $\Gamma \cap (t = 0) = \{\msp\}$.
This shows that $\Gamma$ is irreducible and reduced.
\end{proof}

By our choice of coordinates, $y^2 z$ is the unique monomial in $F_1$ divisible by $y^2$ and we see that monomials of degree $16$ which is not divisible by $y^2$ has initial weight at least $6/5$.
It follows that $\ord_E (z) \ge 6/5$ and $\varphi$ is realized as the embedded weighted blowup at $\msp$ with weight $\wt (x,z,t,u,v,w) = \frac{1}{5} (1,6,2,3,4,5) =: \mbfw$.
By looking at the monomials in $F_1|_{\Pi}$, $F_4|_{\Pi}$, $F_5|_{\Pi}$, the lowest weight parts of $F_1|_{y=1}$, $F_4|_{y=1}$ and $F_5|_{y=1}$ are of the form
\[
F^{\mbfw}_1 = z + v t + u^2 + f, \ 
F^{\mbfw}_4 = v + t^2 + g, \ 
F^{\mbfw}_5 = w + u t + h,
\]
where $f,g,h \in \mbC [x,z,t,u,v,w]$ vanish along $(x = z = 0)$. 
Thus we have an isomorphism
\[
E \cong (z + v t + f = v + t^2 + g = w + u t + h = 0) \subset \mbP (1_x,6_z,2_t,3_u,4_v,5_w).
\]

\begin{Lem}
The singular point of type $\frac{1}{5} (1,2,3)$ is not a maximal centre.
\end{Lem}

\begin{proof}
We claim $\tilde{S} \cap \tilde{T} = \tilde{\Gamma}$.
To see this, it is enough to see that $\tilde{S} \cap \tilde{T}$ does not contain a curve on $E$.
The lift of the sections $x$ and $z$ on $Y$ restricts to the coordinates $x$ and $z$ of the ambient weighted projective space of $E$ and their zero loci coincides with $\tilde{S} \cap E$ and $\tilde{T} \cap E$, respectively.
Since $f,g,h$ are in the ideal $(x,z)$, the set
\[
\tilde{S} \cap \tilde{T} \cap E = (x = z = v t + u^2 = v + t^2 = w + u t = 0)
\] 
consists of a single point.
Thus $\tilde{S} \cap \tilde{T} = \tilde{\Gamma}$.
Since $\tilde{S} \sim_{\mbQ} \varphi^*A - \frac{1}{5} E$ and $\tilde{T} \sim_{\mbQ} 6\varphi^*A - \frac{6}{5} E$, we have
\[
(\tilde{T} \cdot \tilde{S} \cdot \tilde{T}) = 6^2 (A^3) - \frac{6^2}{5^3} (E^3) = \frac{6}{7} - \frac{6}{5} < 0.
\]
Therefore, $\msp$ is not a maximal centre by Lemma \ref{lem:exclbadC}.
\end{proof}

\section{Pfaffian Fano $3$-fold of degree $1/30$} \label{sec:deg30}

Let $X = X_{14,15,16,17,18} \subset \mbP (1_x,5_{y_0},5_{y_1},6_z,7_t,8_u,9_v)$ be a Pfaffian Fano $3$-fold of degree $1/30$.
We exclude all the singular points on $X$ and prove that $X$ is birationally super-rigid under a suitable generality condition.
The syzygy matrix of $X$ and the defining polynomials are given as follows:
\[
M =
\begin{pmatrix}
0 & a_5 & a_6 & a_7 & a_8 \\
& 0 & b_7 & b_8 & b_9 \\
& & 0 & c_9 & c_{10} \\
& & & 0 & d_{11} \\
& & & & 0
\end{pmatrix}
\hspace{1.5cm}
\begin{aligned}
F_1 &= a_5 c_9 - a_6 b_8 + a_7 b_7 \\  
F_2 & = a_5 c_{10} - a_6 b_9 + a_8 b_7 \\ 
F_3 &= a_5 d_{11} - a_7 b_9 + a_8 b_8 \\
F_4 &= a_6 d_{11} - a_7 c_{10} + a_8 c_9 \\
F_5 &= b_7 d_{11} - b_8 c_{10} + b_9 c_9
\end{aligned}
\]
The basket of singularities of $X$ is as follows
\[
\left\{ \frac{1}{5} (1,1,4), 2 \times \frac{1}{5} (1,2,3), \frac{1}{6} (1,1,5)  \right\}
\]

The aim of this section is to prove the following theorem, which will follow from Propositions \ref{prop:exclcurve}, \ref{prop:exclnspt} and the results of the present section.
The condition in the statement will be introduced later.

\begin{Thm}
Let $X$ be a Pfaffian Fano $3$-fold of degree $1/30$.
If $X$ satisfies \emph{Condition \ref{cd:deg30-5}}, then it is birationally super-rigid.
\end{Thm}

\subsection{Exclusion of the $\frac{1}{5} (1,2,3)$ points}

Let $\msp$ be a point of type $\frac{1}{5} (1,2,3)$.
After replacing $y_0,y_1$, we assume $\msp = \msp_{y_1}$.
Note that this implies $y_1^3 \notin F_2$.
Note also that $t^2 \in F_1$, $u^2 \in F_3$ and $v^2 \in F_5$ since $\msp_t, \msp_u,\msp_v \notin X$, which implies $t \in a_7, b_7$, $u \in a_8, b_8$ and $v \in b_9,c_9$.
By quasi-smoothness of $X$ at $\msp$, we have $y_1 v \in F_1$ and $y_1^2 y_0 \in F_2$.
We divide the proof into two cases according to $y_1^2 z \in F_3$ or not.

First, we treat the case where $y_1^2 z \in F_3$.

\begin{Lem}
If $y_1^2 z \in F_3$, then $\msp$ is not a maximal centre.
\end{Lem}

\begin{proof}
Recall that $y_1 v \in F_1$, $y_1^2 y_0 \in F_2$ and $y_1^2 z \in F_3$.
By Lemma \ref{lem:elim}, we may assume that $y_1^2 z$ is the unique monomial in $F_3$ divisible by $y_1^2$.
Consider the weight $\wt (x,y_0,z,t,u,v) = \frac{1}{5} (1,5,6,2,3,4) =: \mbfw$.
Then $v \in F^{\mbfw}_1, y_0 \in F^{\mbfw}_2$ and $z \in F^{\mbfw}_3$ so that $\varphi$ is realized as the embedded weighted blowup at $\msp$ with the weight $\mbfw$. 

We claim that $\{x,y_0,z\}$ isolates $\msp$.
Set $\Pi = (x = y_0 = z = 0)$.
We have
\[
F_1|_{\Pi} = t^2 + \alpha v y_1, F_3|_{\Pi} = u^2 + \beta v t, F_5|_{\Pi} = v^2 + \gamma u y_1^2,
\]
for some $\alpha,\beta,\gamma \in \mbC$.
Hence
\[
(x = y_0 = z = 0) \cap X \subset (x = y_0 = z = F_1|_{\Pi} = F_3|_{\Pi} = F_5|_{\Pi} = 0)
\]
and it is straightforward to see that the set on the right-hand side is finite (for any $\alpha,\beta,\gamma$).
This shows that $\{x,y_0,z\}$ isolates $\msp$.
We see $\ord_E (x,y_0,z) \ge \frac{1}{5} (1, 5, 6)$ so that $L: = B$ is nef by Lemma \ref{lem:isolnef} and we compute
\[
(L \cdot B^2) = (B^3) = (A^3) - \frac{1}{5^3} (E^3) = \frac{1}{30} - \frac{1}{30} = 0.
\]
Therefore, $\msp$ is not a maximal centre by Lemma \ref{lem:excltc}.
\end{proof}

Next, we treat the case where $y_1^2 z \notin F_3$.
In this case, we have $y_1^3 x \in F_3$.
We set $S = (x_0 = 0) \cap X$, $T = (y_0 = 0) \cap X$ and $\Gamma = S \cap T$.

\begin{Lem}
The support of $\Gamma$ is an irreducible curve. 
\end{Lem}

\begin{proof}
We set $\Pi = (x_0 = y_0 = 0)$.
We have $y_1 \in a_5$ because otherwise $F_3 = a_5 d_{11} - a_7 b_9 + a_8 b_8$ cannot contain $y_1^3 x$.
Then, we see $y_1^2 \notin c_{10}$ since $y_1^3 \notin F_2$.
Note also that $z y_1 \notin d_{11}$ since $y_1^2 z \notin F_3$.
We can write the restrictions of the syzygy matrix and defining polynomials to $\Pi$ as
\[
M|_{\Pi} =
\begin{pmatrix}
0 &y_1 & \alpha z & \beta t & \gamma u \\
& 0 & t & u & \delta v \\
& & 0 & v & 0 \\
& & & 0 & 0 \\
& & & & 0
\end{pmatrix}
\hspace{1.5cm}
\begin{aligned}
F_1|_{\Pi} &= y_1 v - \alpha z u + \beta t^2 \\
F_2|_{\Pi} &= - \alpha \delta z v + \gamma u t \\
F_3|_{\Pi} &= - \beta \delta v t + \gamma u^2 \\
F_4|_{\Pi} &= \gamma \delta u v \\
F_5|_{\Pi} &= \delta v^2
\end{aligned}
\]
Note that $\Gamma = X \cap \Pi$ is defined in $\Pi$ by the above $5$ polynomials.
Since $\beta, \gamma, \delta \ne 0$, we have $\Gamma = (t = u = v = 0) \cap \Pi$ set-theoretically and the proof is completed.
\end{proof}

\begin{Lem}
If $y_1^2 z \notin F_3$, then $\msp$ is not a maximal centre.
\end{Lem}

\begin{proof}
We will show that the support of $\tilde{S} \cap \tilde{T}$ is the proper transform of the support of $S \cap T$.
Consider the weight $\wt (x,y_0,z,t,u,v) = \frac{1}{5} (6,5,1,2,3,4) =: \mbfw$.
Then $v \in F^{\mbfw}_1, y_0 \in F^{\mbfw}_2$ and $x \in F^{\mbfw}_3$ since $F_3$ does not contain $y_1^2 z$ which is the unique monomial of degree $16$ with $\mbfw$-weight $\frac{1}{5}$.
It follows that $\varphi$ is realized as the embedded weighted blowup at $\msp$ with weight $\mbfw$ and we have an isomorphism
\[
E \cong (F^{\mbfw}_1 = F^{\mbfw}_2 = F^{\mbfw}_3 = 0) \subset \mbP (6_x,5_{y_0},1_z,2_t,3_u,4_v).
\]
In view of the description of $F_1|_{\Pi}, F_2|_{\Pi}, F_3|_{\Pi}$, after re-scaling $t,u$, we can write
\[
F^{\mbfw}_1 = v + \alpha z u + t^2 + f, \ 
F^{\mbfw}_2 = x + \beta v z + \gamma u t + g, \ 
F^{\mbfw}_3 = y_0 + \delta v t + u^2 + h,
\]
where $\alpha,\dots,\delta \in \mbC$ with $\gamma, \delta \ne 0$ and $f,g,h$ are contained in the ideal $(x,y_0)$ (Note that $x \notin g$ and $y_0 \notin h$).
We have
\[
\tilde{S} \cap \tilde{T} \cap E 
= (x = y_0 = 0) \cap E  
= (x = y_0 = v + \alpha z u + t^2 = \beta v z + \gamma u t = \delta v t + u^2 = 0)
\]
and this is a finite set of points since $\gamma, \delta \ne 0$.
Thus, $\tilde{\Gamma} \cap \tilde{S}$ is the proper transform of $S \cap T$.

We have $\tilde{S} \sim_{\mbQ} \varphi^*A - \frac{6}{5} E$ and $\tilde{T} \sim_{\mbQ} 5 \varphi^*A - \frac{5}{5} E$ so that
\[
(\tilde{T} \cdot \tilde{S} \cdot \tilde{T}) = 5^2 (A^3) - \frac{5^2 \cdot 6}{5^3} (E^3) 
= \frac{5}{6} - 5 < 0.
\]
Therefore, $\msp$ is not a maximal centre by Lemma \ref{lem:exclbadC}.
\end{proof}

\subsection{Exclusion of the $\frac{1}{6} (1,1,5)$ point}

\begin{Lem}
The point of type $\frac{1}{6} (1,1,5)$ is not a maximal centre.
\end{Lem}

\begin{proof}
We claim that $\{x,y_0,y_1\}$ isolates the $\frac{1}{6} (1,1,5)$ point $\msp = \msp_z$.
Set $\Pi = (x = y_0 = y_1 = 0)$.
Then we can write
\[
F_1|_{\Pi} = \alpha u z + \beta t^2, \ 
F_3|_{\Pi} = \gamma v t + \delta u^2, \ 
F_5|_{\Pi} = \varepsilon v^2,
\]
for some $\alpha,\beta,\dots,\varepsilon \in \mbC$.
Moreover, none of $\beta,\delta,\varepsilon$ is zero since $\msp_t,\msp_u,\msp_v \notin X$.
Hence
\[
(x = y_0 = y_1 = 0) \cap X \subset (x = y_0 = y_1 = t = u = v = 0) = \{\msp\},
\]
that is, $\{x,y_0,y_1\}$ isolates $\msp$.

It is clear that $\ord_E (x,y_0,y_1) \ge \frac{1}{6} (1, 5, 5)$ since $x,y_0,y_1$ are of degrees $1,5,5$, respectively (see Lemma \ref{lem:ordE} (1)), so that $L = \varphi^*A - \frac{1}{6} E$ is nef by Lemma \ref{lem:isolnef}.
We compute
\[
(L \cdot B^2) = (B^3) = (A^3) - \frac{1}{6^3} (E^3) = \frac{1}{30} - \frac{1}{30} = 0.
\]
Therefore, $\msp$ is not a maximal centre by Lemma \ref{lem:excltc}.
\end{proof}

\subsection{Exclusion of the $\frac{1}{5} (1,1,4)$ point}

Let $\msp \in X$ be the point of type $\frac{1}{5} (1,1,4)$.
After replacing $y_0,y_1$, we may assume $\msp = \msp_{y_1}$.
We have $t^2 \in F_1$, $u^2 \in F_3$ and $v^2 \in F_5$ since $\msp_t,\msp_u,\msp_v \notin X$, which implies $t \in a_7, b_7$, $u \in a_8,b_8$ and $v \in b_9, c_9$.
Since $\msp$ is of type $\frac{1}{5} (1,1,4)$, we have $v y_1 \notin F_1$ and $y_1^2 z \notin F_3$, which implies $y_1 \notin a_5$.
Since $X$ has a point of type $\frac{1}{6} (1,1,5)$ at $\msp_z$, we have $v z \in F_2$, which implies $z \in a_6$.
Since $X$ has a single point of type $\frac{1}{5} (1,1,4)$ and two distinct points of type $\frac{1}{5} (1,2,3)$, the set $(x = z = t = u = v = a_5 c_{10} = 0)$ consists of three distinct points.
This implies $y_1^2 \in c_{10}$ since $y_1 \notin a_5$.
We set $\Pi = (x_0 = y_0 = 0)$.
Then the restrictions of the syzygy matrix and the defining polynomials can be written as follows:
\[
M|_{\Pi} =
\begin{pmatrix}
0 & 0 & z & \alpha t & \beta u \\
 & 0 & t & u & v \\
& & 0 & \gamma v & y_1^2 \\
& & & 0 & \delta z y^2_1 \\
& & & & 0
\end{pmatrix}
\hspace{1.5cm}
\begin{aligned}
F_1|_{\Pi} &= \alpha t^2 - z u \\ 
F_2|_{\Pi} &= \beta u t - z v \\ 
F_3|_{\Pi} &= \beta u^2 - \alpha t v \\
F_4|_{\Pi} &= \beta \gamma u v - \alpha t y_1^2 + \delta z^2 y_1 \\
F_5|_{\Pi} &= \gamma v^2 - u y_1^2 + \delta t z y_1
\end{aligned}
\]
where $\alpha,\beta,\gamma \in \mbC \setminus \{0\}$ and $\delta \in \mbC$.
We set $S = (x_0 = 0) \cap X$, $T = (y_0 = 0) \cap X$ and let $\Gamma = S \cap T$ be the scheme-theoretic intersection.
We assume the following condition.

\begin{Cond} \label{cd:deg30-5}
Under the above choice of coordinates, $\delta \ne 0$.
\end{Cond}

\begin{Lem}
$\Gamma$ is an irreducible and reduced curve.
\end{Lem}

\begin{proof}
The curve $\Gamma = X \cap \Pi$ is defined by $F_1|_{\Pi} = \cdots = F_5|_{\Pi} = 0$ in $\Pi$.
Recall that $\alpha,\beta,\gamma \ne 0$.
We work on the open subset on which $z \ne 0$.
By setting $z = 1$ in $F_1|_{\Pi} = F_2|_{\Pi} = 0$, we have $u = \alpha t^2$ and $v = \beta ut = \alpha \beta t^3$.
By eliminating $u$ and $v$ in the equation $F_3|_{\Pi} = F_4|_{\Pi} = F_5|_{\Pi} = 0$, we see that, on $z \ne 0$, $\Gamma$ is isomorphic to the quotient of the curve
\[
(\delta y_1 - \alpha y_1^2 t + \alpha^2 \beta^2 \gamma t^5 = 0) \subset \mbA^2_{y_1,t}
\]
by the natural $\mbZ_5$-action on $\mbA^2$.
On the other hand, we have $\Gamma \cap (z = 0)$ consists of the single point $\msp$.
Therefore, $\Gamma$ is an irreducible and reduced curve.
\end{proof}

\begin{Lem}
The point of type $\frac{1}{5} (1,1,4)$ is not a maximal centre.
\end{Lem}

\begin{proof}
We have $y_1^2 y_0 \in F_2$, $y_1^2 t \in F_4$ and $y_1^2 u \in F_5$ by quasi-smoothness of $X$ at $\msp$.
Consider the initial weight $\wt (x,y_0,z,t,u,v) = \frac{1}{5} (1,5,1,2,3,4) = \iniw$.
In view of the description of $F_1|_{\Pi}$, after re-scaling coordinates, we have
\[
F^{\iniw}_2 = y_0 + u t + z v + f, \ 
F^{\iniw}_4 = t + \delta_1 z^2 + g, \ 
F^{\iniw}_5 = u + \delta_2 t z + h, \ 
\]
where $\delta_1, \delta_2 \in \mbC \setminus \{0\}$ and $f,g,h$ are contained in the ideal $(x_0,y_0)$ (Note that $y_0 \notin f$).
We see that $\varphi$ is realized as the embedded weighted blowup at $\msp$ with weight $\iniw$ and we have an isomorphism
\[
E \cong (F^{\iniw}_2 = F^{\iniw}_4 = F^{\iniw}_5 = 0) \subset \mbP (1_x,5_{y_0},1_z,2_t,3_u,4_v).
\]
We see that
\[
\tilde{S} \cap \tilde{T} \cap E = (x = y_0 = ut - zv = t + \delta_1 z^2 = u + \delta_2 t z = 0)
\]
is finite a finite set, which imlies $\tilde{S} \cap \tilde{T} = \tilde{\Gamma}$.

We have $\tilde{S} \sim_{\mbQ} \varphi^*A - \frac{1}{5} E$ and $\tilde{T} \sim_{\mbQ} 5 \varphi^*A - \frac{5}{5} E$, so that
\[
(\tilde{T} \cdot \tilde{S} \cdot \tilde{T}) = 5^2 (A^3) - \frac{5^2}{5^3} (E^3) = \frac{5}{6} - \frac{5}{4} < 0.
\]
Therefore, $\msp$ is not a maximal centre by Lemma \ref{lem:exclbadC}.
\end{proof}

\section{Pfaffian Fano $3$-fold of degree $1/20$} \label{sec:deg20}

Let $X = X_{12,13,14,15,16} \subset \mbP (1_x,4_y,5_{z_0},5_{z_1},6_t,7_u,8_v)$ be a Pfaffian Fano $3$-fold of degree $1/20$.
We exclude singular points on $X$ other than the $\frac{1}{5} (1,2,3)$ point at which there is a birational involution and prove that $X$ is birationally rigid under a suitable generality condition.
The syzygy matrix of $X$ and the defining polynomials are given as follows:
\[
M =
\begin{pmatrix}
0 & a_4 & a_5 & a_6 & a_7 \\
& 0 & b_6 & b_7 & b_8 \\
& & 0 & c_8 & c_9 \\
& & & 0 & d_{10} \\
& & & & 0
\end{pmatrix}
\hspace{1.5cm}
\begin{aligned}
F_1 &= a_4 c_8 - a_5 b_7 + a_6 b_6  \\
F_2 & = a_4 c_9 - a_5 b_8 + a_7 b_6 \\ 
F_3 &= a_4 d_{10} - a_6 b_8 + a_7 b_7 \\ 
F_4 &= a_5 d_{10} - a_6 c_9 + a_7 c_8 \\
F_5 &= b_6 d_{10} - b_7 c_9 + b_8 c_8
\end{aligned}
\]
The basket of singularities of $X$ is as follows
\[
\left\{ \frac{1}{2} (1,1,1), \frac{1}{4} (1,1,3), 2 \times \frac{1}{5} (1,1,4), \frac{1}{5} (1,2,3) \right\}.
\]

The aim of this section is to prove the following theorem, which will follow from Propositions \ref{prop:exclcurve}, \ref{prop:exclnspt} and the results of the present section.
The condition in the statement will be introduced later.

\begin{Thm}
Let $X$ be a Pfaffian Fano $3$-fold of degree $1/20$.
If $X$ satisfies \emph{Condition \ref{cd:deg20-4}}, then it is birationally rigid.
\end{Thm}

\subsection{Exclusion of the $\frac{1}{2} (1,1,1)$ point}

\begin{Lem}
The singular point of type $\frac{1}{2} (1,1,1)$ is not a maximal centre.
\end{Lem}

\begin{proof}
Let $\msp$ be the point of type $\frac{1}{2} (1,1,1)$.
It is clear that $\{ x,z_0,z_1,u\}$ isolates $\msp$ and $\ord_E (x,z_0,z_1,u) \ge \frac{1}{2} (1,1,1,1)$.
It follows that $L = 7 \varphi^*A - \frac{1}{2} E$ is nef by Lemma \ref{lem:isolnef} and we compute
\[
(L \cdot B^2) = 7 (A^3) - \frac{1}{2^3} (E^3) = \frac{7}{20} - \frac{1}{2} < 0.
\]
Therefore, $\msp$ is not a maximal centre by Lemma \ref{lem:excltc}.
\end{proof}

\subsection{Exclusion of the $\frac{1}{5} (1,1,4)$ points}

\begin{Lem}
A singular point of type $\frac{1}{5} (1,1,4)$ is not a maximal centre.
\end{Lem}

\begin{proof}
Let $\msp$ be a point of type $\frac{1}{5} (1,1,4)$.
We may assume $\msp = \msp_{z_1}$ after replacing $z_0,z_1$.
We claim that $\{x,y,z_0\}$ isolates $\msp$.
Set $\Pi = (x = y = z_0 = 0)$.
Note that $t^2 \in F_1$, $u^2 \in F_3$, $v^2 \in F_5$ since $\msp_t, \msp_u, \msp_v \notin X$, hence we may assume that those coefficients are $1$.
Then we can write
\[
F_1|_{\Pi} = t^2 + \alpha u z_1, \ 
F_3|_{\Pi} = u^2 + \beta v t, \ 
F_5|_{\Pi} = v^2 + \gamma t z_1^2,
\]
for some $\alpha,\beta,\gamma \in \mbC$.
We see that
\[
(x = y = z_0 = 0) \cap X \subset (x = y = z_0 = F_1|_{\Pi} = F_2|_{\Pi} = F_3|_{\Pi} = 0)
\]
and the set in the right-hand side of the above equation is finite (for any $\alpha,\beta,\gamma \in \mbC$).
This shows that $\{x,y,z_0\}$ isolates $\msp$.

We see $\ord_E (x,y,z_0) \ge \frac{1}{5} (1,4,5)$ so that $L = B$ is nef by Lemma \ref{lem:isolnef}.
We compute
\[
(L \cdot B^2) = (B^3) = (A^3) - \frac{1}{5^3} (E^3) = \frac{1}{20} - \frac{1}{20} = 0.
\]
Therefore, $\msp$ is not a maximal centre by Lemma \ref{lem:excltc}.
\end{proof}

\subsection{Exclusion of the $\frac{1}{4} (1,1,3)$ point}

Let $\msp$ be a point of type $\frac{1}{4} (1,1,3)$.
Replacing $v$, we assume $\msp = \msp_y$.
We claim $y \in a_4$.
Indeed, if $y \notin a_4$, then $z_0^2 y, z_0 z_1 y, z_1^2 y \notin F_3$.
This is a contradiction since $X$ admits a point $\msq$ of type $\frac{1}{5} (1,2,3)$, hence there must be at least one of $z_0^2 y$, $z_0 z_1 y$ and $z_1^2 y$ in $F_3$.
Hence $y \in a_4$ and we assume that the coefficient of $y$ in $a_4$ is $1$ after re-scaling $y$.
We can write $a_5 = \ell_1 + (\text{other terms})$, $c_9 = y \ell_2 + \text{other terms})$ and $d_{10} = q + (\text{other terms})$, where $\ell_1,\ell_2$ are linear forms in $z_0,z_1$ and $q$ is the quadratic form in $z_0,z_1$.
Let $\delta \in \mbC$ be the coefficient of $y^2 \in b_8$.
We exclude $\msp$ assuming  the following:

\begin{Cond} \label{cd:deg20-4}
Under the above choice of coordinates, the polynomials $\ell_2 - \delta \ell_1$ and $q$ have no (non-trivial) common root. 
\end{Cond}

We have $F_2 = y^2 (\ell_2 - \delta \ell_1) + (\text{other terms})$.
Condition \ref{cd:deg20-4} in particular implies $\ell_2 - \delta \ell_1 \ne 0$.
Replacing $z_0, z_1$, we assume $\ell_2 - \delta \ell_1 = z_1$.
This means that $y^2 z_1 \in F_2$ and $y^2 z_0 \notin F_2$.
By Lemma \ref{lem:elim}, replacing $z_1$ further, we may assume that $y^2 z_1$ is the unique monomial in $F_2$ divisible by $y^2$.
We have $t^2 \in F_1$, $u^2 \in F_3$ and $v^2 \in F_5$ since $\msp_t, \msp_u, \msp_v \notin X$, which implies $t \in a_6,b_6$, $u \in a_7,b_7$ and $v \in b_8,c_8$.
By setting $\Pi = (x = z_1 = 0)$, we can write
\[
M|_{\Pi} =
\begin{pmatrix}
0 & y & \alpha z_0 & t & \beta u \\
& 0 & \gamma t & u & v + \delta y^2 \\
& & 0 & \varepsilon v + \eta y^2 & \zeta z_0 y \\
& & & 0 & \lambda t y + \mu z_0^2 \\
& & & & 0
\end{pmatrix},
\]
where $\alpha,\beta,\cdots,\mu \in \mbC$.
Note that $\beta, \gamma,\varepsilon \ne 0$.
Note also that $\alpha$ is the coefficient of $z_0$ in $\ell_1$ and $\mu$ is the coefficient of $z_0^2$ in $q$.
We have $\mu \ne 0$ because otherwise $\ell_2 - \delta \ell_1 = q = 0$ has a solution $z_1 = 0$ and this is impossible by Condition \ref{cd:deg20-4} (Here, recall that $\ell_2 - \delta \ell_1 = z_1$). 
Since $\msp = \msp_y \in X$ and the coefficient of $y^3$ in $F_1|_{\Pi}$ is $\eta$, we have $\eta = 0$.
The coefficient of $y^2 z_0$ in $F_2|_{\Pi}$ is $\alpha (\zeta - \delta)$ which must be $0$ by our choice of coordinates.
Thus, we have
\[
\begin{split}
F_1|_{\Pi} &= \varepsilon y v - \alpha u z_0 + \gamma t^2, \\ 
F_2|_{\Pi} &= \beta \gamma u t - \alpha v z_0, \\
F_3|_{\Pi} &= (\lambda - \delta) t y^2 + \mu z_0^2 y - v t + \beta u^2, \\
F_4|_{\Pi} &= (\alpha \lambda - \zeta) t z_0 y + \alpha \mu z_0^3 + \beta \varepsilon v u, \\
F_5|_{\Pi} &= \gamma \lambda t^2 y + \gamma \mu t z_0^2 - \zeta u z_0 y + \varepsilon v^2 + \delta \varepsilon v y^2.
\end{split}
\]
By quasi-smoothness of $X$ at $\msp$, we have $\lambda - \delta \ne 0$.
We compute $\ord_E (z_1)$.
We see that $y^3 x$, $y^2 z_0$ and $y^2 z_1$ are the monomials of degree $13$ which have initial weight $1/4$ and $y^3 x, y^2 z_0 \notin F_2$ by our choice of coordinates, hence $\ord_E (z_1) \ge 5/4$.
It follows that $\varphi$ is realized as the embedded weighted blowup at $\msp$ with $\wt (x,z_0,z_1,t,u,v) = \frac{1}{4} (1,1,5,2,3,4) =: \mbfw$.

We first consider the general case $\alpha \ne 0$.
Set $S = (x=0) \cap X$, $T = (z_1 = 0) \cap X$ and $\Gamma = S \cap T = \Pi \cap X$.

\begin{Lem}
If $\alpha \ne 0$, then $\Gamma$ is an irreducible and reduced curve.
\end{Lem}

\begin{proof}
In this case we have $\zeta = \delta$ since $\alpha (\zeta - \delta) = 0$.
We work on the open subset $U = (z_0 \ne 0) \subset \Pi$ by setting $z_0 = 1$.
Re-scaling $z_0$, we may assume $\alpha = 1$.
By $F_2|_{\Pi} = 0$, we have $v = \beta \gamma u t$.
For a polynomial $F = F (x,y,z_0,z_1,t,u,v)$, we set $\bar{F} = F (0,y,1,0,u,\beta \gamma u t)$.
Then, by eliminating $v$, we see that $\Gamma \cap U$ is the quotient of affine scheme defined by the polynomials
\[
\begin{split}
f_1 &:= \bar{F}_1 = \beta \gamma \varepsilon u t y - u + \gamma t^2, \\
f_3 &:= \bar{F}_3 = (\lambda - \delta) t y^2 + \mu y - \beta \gamma u t^2 + \beta u^2, \\
f_4 &:= \bar{F}_4 = (\lambda - \delta) t y + \mu + \beta^2 \gamma \varepsilon u^2 t, \\
f_5 &:= \bar{F}_5 = \gamma \lambda t^2 y + \gamma \mu t - \delta u y + \beta^2 \gamma^2 \varepsilon u^2 t^2 + \beta \gamma \delta \varepsilon u t y^2.
\end{split}
\]
in $\mbA^3_{y,t,u}$.
We define
\[
\Delta = (f_1 = f_3 = f_4 = f_5 = 0) \subset \mbA^3_{y,u,t}.
\]
We have $f_3 = y f_4 - \beta u f_1$ and $f_5 = \gamma t f_4 + \delta y f_1$, which implies that $\Delta$ is defined by $f_1 = f_4 = 0$.
Set
\[
\theta = \frac{\beta \gamma \varepsilon}{\lambda - \delta} \ne 0.
\] 
and we eliminate the term $u t y$ from $f_1$, that is, we consider $f'_1 = f_1 - \theta f_4$.
Then $\Delta$ is defined by $f'_1 = f_4 = 0$.
Here we have $f_1' = \theta_1 u + \gamma t^2 + \theta_2 u^3 t$, where $\theta_1 = - (\theta \mu + 1)$ and $\theta_2 = - \beta^2 \gamma \varepsilon \theta$.
Note that $\theta_1$ can be $0$ while $\theta_2 \ne 0$.
We have $(t = 0) \cap \Delta = \emptyset$ since $\mu \ne 0$.
It follows that $\Delta$ is contained in the open subset $(t \ne 0) \subset \mbA^3$.
The projection $\mbA^3_{y,t,u} \ratmap \mbA^2_{t,u}$ induces an isomorphism $\Delta \to \Xi \cap (t \ne 0)$, where $\Xi$ is the curve in $\mbA^2_{y,u}$ defined by $f'_1 = 0$.
If $\theta_1 \ne 0$, it is clear that $\Xi$ is irreducible and reduced, and so is $\Delta$.
If $\theta_1 = 0$, then $f'_1 = t (\gamma t + \theta_2 u^3)$ and $\Xi \cap (t \ne 0)$ is defined by $\gamma t - \theta_2 u^3 = 0$.
Since $\gamma \ne 0$, $\Xi \cap (t \ne 0)$ is irreducible and reduced, and so is $\Delta$.
Therefore, $\Delta$ is irreducible and reduced, and so is $\Gamma \cap U$.

We consider $\Gamma \cap (z_0 = 0)$.
Since $F_2|_{\Pi} = \beta \gamma u t - v z_0$, we have $\Gamma \cap (z_0 = 0) = \Sigma_1 \cup \Sigma_2$, where $\Sigma_1 = \Gamma \cap (z_0 = t = 0)$ and $\Sigma_2 = \Gamma \cap (z_0 = u = 0) \cap (t \ne 0)$.
It is easy to see $\Sigma_t = \{\msp_y\}$.
We have
\[
\begin{split}
\Sigma_2 &= (z_0 = u = \varepsilon y v + \gamma t^2 = (\lambda - \delta) y^2 - v = \gamma \lambda t^2 y + \varepsilon v^2 + \delta \varepsilon v y^2 = 0) \cap (t \ne 0) \\
& = (z_0 = u = \varepsilon y v + \gamma t^2 = (\lambda - \delta) y^2 - v = 0) \cap (t \ne 0)
\end{split}
\] 
and it is straightforward to see that $\Sigma_2$ consists of $2$ points.
Therefore, $\Gamma$ is an irreducible and reduced curve.
\end{proof}

\begin{Lem}
If $\alpha \ne 0$, then $\msp$ is not a maximal centre.
\end{Lem}

\begin{proof}
We will show $\tilde{S} \cap \tilde{T} = \tilde{\Gamma}$.
We have an isomorphism
\[
E \cong (F^{\mbfw}_1 = F^{\mbfw}_2 = F^{\mbfw}_3 = 0) \subset \mbP (1_x, 1_{z_0}, 5_{z_1}, 2_t,3_u,4_v).
\]
Note that $F^{\mbfw}_i|_{x=z_1=0}$ coincides with the lowest weight part of $(F_i|_{\Pi})|_{y=1}$.
Hence we have
\[
F^{\mbfw}_1 = \varepsilon v - u z_0 + \gamma t^2 + f, \ 
F^{\mbfw}_2 = \beta \gamma ut - v z_0 + g, \ 
F^{\mbfw}_3 = (\lambda - \delta) t + \mu z_0^2 + h,
\]
where $f,g,h \in (x,z_1)$.
It is straightforward to see 
\[
\tilde{S} \cap \tilde{T} \cap E = (x = z_1 = F^{\mbfw}_1 = F^{\mbfw}_2 = F^{\mbfw}_3 = 0)
\]
is a finite set of points, which implies $\tilde{S} \cap \tilde{T} = \tilde{\Gamma}$.

Finally, since $\tilde{S} \sim_{\mbQ} \varphi^*A - \frac{1}{4} E = B$ and $\tilde{T} \sim_{\mbQ} 5 \varphi^*A - \frac{5}{4} E = 5 B$, we have
\[
(\tilde{T} \cdot \tilde{S} \cdot \tilde{T}) = 5^2 (A^3) - \frac{5^2}{4^3} (E^3) = \frac{5}{4} - \frac{5^2}{12} < 0.
\]
Therefore, $\msp$ is not a maximal centre by Lemma \ref{lem:exclbadC}.
\end{proof}

Next, we consider the case $\alpha = 0$.

\begin{Lem}
If $\alpha = 0$, then $\msp$ is not a maximal centre.
\end{Lem}

\begin{proof}
We see that $y^3 x^2$, $y^2 t$ and $y z_0^2$ are the only monomials of degree $14$ having $\mbfw$-weight $\frac{2}{4}$.
Note that the coefficients of $t y^2$ and $z_0^2 y$ in $F_3$ are $\lambda - \delta$ and $\mu$, respectively, and let $\theta$ be the coefficient of $y^3 x^2$ in $F_3$.
We set $s = \theta y^2 x^2 + (\lambda - \delta) t y + \mu z_0^2$.
Since the monomials in $F_3$ other than $y^3 x^2$, $t y^2$ and $z_0^2 y$ have $\mbfw$-weight greater than $2/4$, we have $\ord_E (s) \ge 6/4$.

We will show that $\{x,z_1,s\}$ isolates $\msp$.
It is enough to show that 
\[
\Sigma := (s = F_1|_{\Pi} = \cdots = F_5|_{\Pi} = 0) \cap \Pi^{\circ} \subset \Pi^{\circ}
\] 
is a finite set of points, where $\Pi^{\circ} = \Pi \cap (y \ne 0)$.
For a subset $\Xi$ of $\Pi$ and monomials $g_1,\dots,g_k$, we define $\Xi_{g_1,\dots,g_k} = \Xi \cap (g_1 = \cdots = g_k = 0)$.
We claim $\Sigma^{\circ} := \Sigma \cap (u \ne 0) = \emptyset$.
We have $\Sigma^{\circ} = \Sigma^{\circ}_t$ since $F_2|_{\Pi} = \beta \gamma u t$.
Then we see $\Sigma^{\circ} = \emptyset$ since $F_3|_{\Pi} = s|_{\Pi} - v t + \beta u^2$ and $u \ne 0$ on $\Sigma^{\circ}$.
This implies $\Sigma = \Sigma_u$.
We have $F_3|_{\Pi'} = s|_{\Pi'} - v t$, hence $F_3|_{\Sigma} = - v t$.
Thus $\Sigma = \Sigma_u = \Sigma_{u,v} \cup \Sigma_{u,t}$.
Since $F_1|_{\Pi_u} = \varepsilon y v + \gamma t^2$, we have $\Sigma_{u,v} \subset \Sigma_{u,t}$.
This shows $\Sigma = \Sigma_{u,t}$ and it is defined in $\Pi_{u,t}$ by the equations
\[
\mu z_0^2 
= \varepsilon y v 
= \varepsilon v^2 + \delta \varepsilon v y^2 = 0.
\]
It is now straightforward to see $\Sigma = \{\msp\}$.

Now, since $\ord_E (x,z_1,s) \ge \frac{1}{4} (1,5,6)$, we see that $L = 10\varphi^*A - \frac{6}{4} E$ is nef by Lemma \ref{lem:isolnef} and we have
\[
(L \cdot B^2) = 10 (A-3) - \frac{6}{4^3} (E^3) = \frac{1}{2} - \frac{1}{2} = 0.
\]
Therefore, $\msp$ is not a maximal centre by Lemma \ref{lem:excltc}.
\end{proof}

\subsection{The $\frac{1}{5} (1,2,3)$ point and birational involution}

Let $\msp \in X$ be the point of type $\frac{1}{5} (1,2,3)$.
We assume $\msp = \msp_{z_1}$ after replacing $z_0$ and $z_1$.
We have $u \in a_6, b_6$ and $v \in a_7, b_7$ since $\msp_u, \msp_v \notin X$.
Since $\msp$ is of type $\frac{1}{5} (1,2,3)$, we have $z_1^2 y \in F_3$ and $z_1^2 z_0 \in F_4$.
By $z_1^2 y \in F_3 = a_4 d_{10} - a_6 b_8 + a_7 b_7$, we have $y \in a_4$ and $z_1^2 \in d_{10}$.
It follows that $z_1^2 t \in F_5 = b_6 d_{10} - b_7 c_9 + b_8 c_8$.
Thus $\varphi$ is the weighted blowup with weight $\wt (x,u,v) = \frac{1}{5} (1,2,3)$.
By Lemma \ref{lem:elim}, we can assume that $z_1^2 t$ is the unique monomial in $F_5$ divisible by $z^2$.
We see that $z_1^3 x$ and $z_1^2 t$ are all the monomials of degree $16$ having initial weight $\frac{1}{5}$.
By our choice of coordinates, $z_1^3 x \notin F_5$, hence $\wt (x,y,z_0,t,u,v) = \frac{1}{5} (1,4,5,6,2,3) =: \mbfw$ satisfies the KBL condition.

Let $\pi \colon X \ratmap \mbP := \mbP (1,4,5,6)$ be the projection to the coordinates $x,y,z_0,t$.
We have
\[
F_3 (0,0,0,z_1,0,u,v) = \lambda u^2, \ 
F_5 (0,0,0,z_1,0,u,v) = \mu v^2,
\]
for some $\lambda, \mu \in \mbC \setminus \{0\}$ since $u \in a_6,b_6$ and $v \in a_7 b_7$.
Hence we have $(x = y = z_0 = t = 0) \cap X = \{\msp\}$, which implies that $\pi$ is defined outside $\msp$.
Let $\pi_Y \colon Y \ratmap \mbP$ be the induced rational map.
We take $H \in |\mcO_{\mbP} (1)|$.

\begin{Lem}
The map $\pi_Y$ is a surjective generically finite morphism of degree $2$ such that $B = \pi_Y^*H$.
\end{Lem}

\begin{proof}
First, we show that $\pi_Y$ is everywhere defined.
It is enough to show that $\pi_Y$ is defined at every point of $E$.
We see that $\varphi$ is realized as the embedded weighted blowup at $\msp$ with weight $\mbfw$ and we have an isomorphism
\[
E \cong (F^{\mbfw}_3 = F^{\mbfw}_4 = F^{\mbfw}_5 = 0) \subset \mbP (1_x,4_y,5_{z_0},6_t,2_u,3_v).
\]
The indeterminacy locus of $\pi_Y$ is the set $(x = y = z_0 = t = 0) \cap E$. 
We see that $F^{\mbfw}_3 = y + \alpha u^2 + g_3$, $F^{\mbfw}_4 = z_0 + \beta v u + g_4$ and $F^{\mbfw}_5 = t + \gamma v^2 + g_5$, where $g_3, g_4, g_5 \in (x,y,z_0,t)$, $y \notin g_3$, $z_0 \notin g_4$, $t \notin g_5$ and $\alpha, \beta, \gamma \ne 0$.
Hence, the set $(x = y = z_0 = t = 0) \cap E$ is empty, which shows that $\pi_Y$ is a morphism.

By the construction, $\pi_Y^*H$ is the proper transform of $(x = 0) \cap X$ via $\varphi$, which is $B$ since $\ord_E (x) = 1/5$.
We have $(H^3) = 1/120$ and 
\[
(B^3) = (A^3) - \frac{1}{5^3} (E^3) = \frac{1}{20} - \frac{1}{30} = \frac{1}{60}.
\]
This implies that $\pi_Y$ is a surjective generically finite morphism of degree $2$.
\end{proof}

\begin{Prop} \label{prop:deg20-birinv}
One of the following holds.
\begin{enumerate}
\item $\msp$ is not a maximal centre.
\item There is a birational involution $\sigma \colon X \ratmap X$ which is a Sarkisov link centred at $\msp$.
\end{enumerate}
\end{Prop}

\begin{proof}
We take the Stein factorization of $\pi_Y$ and let $\psi \colon Y \to Z$ be the birational morphism, $\pi_Z \colon Z \to \mbP$ be the double cover such that $\pi_Y = \pi_Z \circ \psi$.
By Lemma \ref{lem:2cover} below, $\psi$ is not an isomorphism.
Thus, by \cite[Lemma 3.2]{OkadaII}, either (1) or (2) happen depending on whether $\psi$ is divisorial or small.
\end{proof}

We use the following result in the above proof.

\begin{Lem} \label{lem:2cover}
Let $X$ be a $\mbQ$-Fano $3$-fold embedded in a weighted projective space $\mbP (a_0,\dots,a_n)$.
Suppose that $X$ is quasi-smooth and let $\varphi \colon Y \to X$ be the Kawamata blowup of $X$ at a terminal quotient singular point $\msp \in X$.
Then $Y$ cannot be a double cover of any weighted projective $3$-space.
\end{Lem}

\begin{proof}
Assume that there is a double cover $\pi \colon Y \to \mbP:= \mbP (b_0,\dots,b_3)$.
Let $D \subset \mbP$ be the branched divisor and $f$ the defining polynomial of $D$.
Then $Y$ is isomorphic to the weighted hypersurface $Z := (y^2 - f = 0) \subset \mbP (b_0,\dots,b_3,d)$, where $2 d = \deg f$ and $d = \deg y$. 
Since $X$ is quasi-smooth and $\varphi$ is a Kawamata blowup, we see that $Y$ has only (terminal) quotient singularities, and so is $Z \cong Y$.
This implies that $Z$ is quasi-smooth and this implies that the Picard number of $Z$ is one (see \cite[Theorem 3.2.4]{Dolgachev}).
This is a contradiction since the Picard number of $Y$ is $2$.
\end{proof}

\section{Pfaffian Fano $3$-fold of degree $1/12$} \label{sec:deg12}

Let $X = X_{10,11,12,13,14} \subset \mbP (1_x,3_y,4_z,5_{t_0},5_{t_1},6_u,7_v)$ be a Pfaffian Fano $3$-fold of degree $1/12$.
The main aim of this section is to prove that there is a Sarkisov link centred at the $\frac{1}{5} (1,2,3)$ point to a Mori fiber space other than $X$.
This implies that $X$ is not birationally rigid.
Unfortunately we are unable to construct an explicit link.
Instead, we will show that the Kawamata blowup at the $\frac{1}{5} (1,2,3)$ admits a flop (and thus there is a link to a Mori fiber space) and then derive a contradiction assuming the target of the link is isomorphic to $X$.
To do this, we need to exclude or untwist the other centres, so we will exclude singular points of type $\frac{1}{3} (1,1,2)$ and construct a Sarkisov link centred at the $\frac{1}{5} (1,1,4)$ point which is a birational involution.
The syzygy matrix of $X$ and the defining polynomials are given as follows:
\[
M =
\begin{pmatrix}
0 & a_3 & a_4 & a_5 & a_6 \\
& 0 & b_5 & b_6 & b_7 \\
& & 0 & c_7 & c_8 \\
& & & 0 & d_9 \\
& & & & 0
\end{pmatrix}
\hspace{1.5cm}
\begin{aligned}
F_1 &= a_3 c_7 - a_4 b_6 + a_5 b_5 \\ 
F_2 & = a_3 c_8 - a_4 b_7 + a_6 b_5 \\ 
F_3 &= a_3 d_9 - a_5 b_7 + a_6 b_6 \\ 
F_4 &= a_4 d_9 - a_5 c_8 + a_6 c_7 \\
F_5 &= b_5 d_9 - b_6 c_8 + b_7 c_7
\end{aligned}
\]
The basket of singularities of $X$ is as follows
\[
\left\{ 2 \times \frac{1}{3} (1,1,2), \frac{1}{4} (1,1,3), \frac{1}{5} (1,1,4), \frac{1}{5} (1,2,3) \right\}.
\]

We have $u \in a_6, b_6$ and $v \in b_7, c_7$ since $\msp_u, \msp_v \notin X$.

\subsection{Exclusion of the $\frac{1}{4} (1,1,3)$ point}

Let $\msp = \msp_z$ be the point of type $\frac{1}{4} (1,1,3)$.
For the entries $a_5, b_5,d_9$ of the syzygy matrix $M$, we write $a_5 = \ell_1 + (\text{other terms})$, $b_5 = \ell_2 + (\text{other terms})$ and $d_9 = z \ell_3 + (\text{other terms})$, where $\ell_i = \ell_i (t_0,t_1)$ is a linear form.
We see that the solutions of $x = y = z = u = v = \ell_1 \ell_2 = 0$ corresponds to the $\frac{1}{5} (1,1,4)$ and $\frac{1}{5} (1,2,3)$, so that $\ell_1 \ne 0$ and $\ell_2$ are not proportional.
We assume $z \in a_4$.
Then we can assume that the coefficient of $z$ in $a_4$ is $1$ by re-scaling $z$ and let $\varepsilon \in \mbC$ be the coefficient of $z^2$ in $c_8$.

\begin{Lem} \label{lem:deg12-4ell}
We have $\ell_3 \ne 0$ and $\ell_1, \ell_3$ are not proportional.
\end{Lem}

\begin{proof}
We have 
\[
F_1 = \ell_1 \ell_2 + \cdots, \ 
F_3 = - \ell_1 v + \cdots, \ 
F_4 = z^2 \ell_3 - \delta z^2 \ell_1 \cdots, \ 
F_5 = z \ell_2 \ell_3 + \cdots.
\]
Let $\msq_1$ and $\msq_2$ be the singular points corresponding to the solutions $\ell_1 = 0$ and $\ell_2 = 0$ respectively.
We see that $\msq_2$ is of type $\frac{1}{5} (1,1,4)$ since $F_3 = \ell_1 v + \cdots$, hence $\msq_1$ is of type $\frac{1}{5} (1,2,3)$.
Assume that $\ell_3 = \theta \ell_1$ for some $\theta \in \mbC$.
Then $F_3 = \theta z \ell_1 \ell_2 + \cdots$ and this implies that $(\prt F/\prt z) (\msq_2) = 0$.
This is a contradiction since $\msq_2$ is of type $\frac{1}{5} (1,1,4)$ and the proof is completed.
\end{proof}

We exclude the point $\msp$ assuming the following:

\begin{Cond} \label{cd:deg12-4}
We have $z \in a_4$ and, under the above choice of coordinates, $\ell_3 - \delta \ell_1 \not\sim \ell_2$.
\end{Cond}

We have $u^2 \in F_3$ and $t^2 \in F_5$ since $\msp_u, \msp_v \notin X$, which implies $u \in a_6, b_6$ and $v \in b_7$ and $c_7$.
We have $F_4 = z^2 (\ell_3 - \varepsilon \ell_1) + \cdots$ and $\ell_3 - \varepsilon \ell_1 \ne 0$ by Lemma \ref{lem:deg12-4ell}. 
Replacing $t_0$ and $t_1$, we may assume $\ell_3 - \delta \ell_1 = t_0$.
By Lemma \ref{lem:elim}, after further replacing $t_0$, we can assume that $z^2 t_0$ is the unique monomial in $F_4$ which is divisible by $z^2$.
Set $\Pi = (x = y = t_0 = 0)$.
Then the restriction of $M$ and defining polynomials on $\Pi$ can be written as follows:
\[
M|_{\Pi} =
\begin{pmatrix}
0 & 0 & z & \alpha t_1 & u \\
& 0 & \beta t_1 & \gamma u & v \\
& & 0 & \delta v & \varepsilon z^2 \\
& & & 0 & \zeta z t_1 \\
& & & & 0
\end{pmatrix}
\hspace{1.5cm}
\begin{aligned}
F_1|_{\Pi} &= - \gamma z u + \alpha \beta t_1^2 \\ 
F_2|_{\Pi} &= - z v + \beta u t_1 \\ 
F_3|_{\Pi} &= - \alpha t_1 v + \gamma u^2 \\
F_4|_{\Pi} &= (\zeta - \alpha \varepsilon) z^2 t_1 + \delta u v \\ 
F_5|_{\Pi} &= \beta \zeta z t_1^2 - \gamma \varepsilon u z^2 + \delta v^2.
\end{aligned}
\]
for some $\alpha,\beta,\dots,\zeta \in \mbC$ with $\gamma, \delta \ne 0$.
By our choice of coordinates, we have $z^2 t_1 \notin F_4$, that is, $\zeta - \alpha \varepsilon = 0$.

\begin{Lem}
The point $\msp$ of type $\frac{1}{4} (1,1,3)$ is not a maximal centre.
\end{Lem}

\begin{proof}
We see $z u \in F_1$, $z v \in F_2$ and $z^2 t_0 \in F_4$.
We see that the $z^3 x$ and $z^2 t_0$ are the only monomials of degree $13$ having initial weight $\frac{1}{4}$.
By our choice of coordinates, we have $z^3 x \notin F_4$.
This implies that the weight $\wt (x,y,t_0,t_1,u,v) = \frac{1}{4} (1,3,5,1,2,3)$ satisfies the KBL condition.

We claim that none of $\alpha$ and $\beta$ is zero.
If $\alpha = 0$, then $\ell_1 \sim t_0$.
Since $\ell_3 - \varepsilon \ell_1 = z_0$, this implies $\ell_3 \sim \ell_1$.
This is impossible.
If $\beta = 0$, then $\ell_2 \sim t_0$ and this is impossible by Condition \ref{cd:deg12-4}.

It is now straightforward to check $X \cap \Pi = \{\msp\}$ since $\alpha,\beta,\gamma,\delta \ne 0$.
In particular, $\{x,y,t_0\}$ isolates $\msp$.
We have $\ord_E (x,y,t_0) \ge \frac{1}{4} (1,3,5)$ so that $L = B$ is nef by Lemma \ref{lem:isolnef} and we compute
\[
(L \cdot B^2) = (A^3) - \frac{1}{4^3} (E^3) = \frac{1}{12} - \frac{1}{12} = 0.
\]
Therefore, $\msp$ is not a maximal centre by Lemma \ref{lem:excltc}.
\end{proof}

\subsection{The point of type $\frac{1}{3} (1,1,2)$}

Let $\msp$ be a point of type $\frac{1}{3} (1,1,2)$.
After replacing coordinates, we assume $\msp = \msp_y$.
We assume $y \in a_3$.
Then, re-scaling $y$, we can assume that the coefficient of $y$ in $a_3$ is $1$.
We see $y v \in F_1$ and, replacing $v$, we assume that $y v$ is the unique monomial in $F_1$ divisible by $y$.
We can write the entries of the syzygy matrix as $a_5 = \ell_1 + (\text{other term})$, $b_5 = \ell_2 + (\text{other terms})$, $c_8 = y \ell_3 + \eta z^2 + (\text{other terms})$ and $d_9 = z \ell_4 + (\text{other terms})$ for some linear forms $\ell_1,\dots,\ell_4$ in $t_0,t_1$ and $\eta \in \mbC$.
Let $\alpha, \beta$ and $\delta$ be the coefficients of $z$, $y^2$ and $z y$ in $a_4$, $a_6$ and $b_7$ respectively.

\begin{Lem}
We have $\ell_1, \ell_2 \ne 0$.
Moreover, $\ell_1 \not\sim \ell_2$, $\ell_1 \not\sim \ell_4$ and $\ell_2 \not\sim \ell_3$.
\end{Lem}

\begin{proof}
The set 
\[
(x = y = z = u = v = 0) \cap X = (x = y = z = u = v = \ell_1 \ell_2 = 0)
\]
consists of two singular points of type $\frac{1}{5} (1,2,3)$ and $\frac{1}{5} (1,1,4)$, which implies $\ell_1 \ell_2 \ne 0$ and $\ell_1 \not\sim \ell_2$.
In this proof, we assume $\ell_1 = t_0$ and $\ell_2 = t_1$ after replacing $t_0$ and $t_1$.
Since $F_3 = a_3 d_9 - a_5 b_7 + a_6 b_6$, $v \in b_7$ and $v \notin a_3, d_9,a_5,a_6,b_6$, we see that $v t_0 \in F_1$ and $v t_1 \notin F_1$.
This shows that $\msp_{t_0}$ and $\msp_{t_1}$ are of type $\frac{1}{5} (1,1,4)$ and $\frac{1}{5} (1,2,3)$, respectively. 

Assume $\ell_3 \sim \ell_2$, that is, $\ell_3 = \nu t_1$ for some $\nu \in \mbC$.
Since $\msp_{t_0}$ is of type $\frac{1}{5} (1,1,4)$, we have $t_0^2 y \in F_4$.
But since $F_4|_{\Pi} = - \ell_1 \ell_3 + \dots$, $\ell_1 = t_0$ and $\ell_3 = \nu t_1$, we see $t_0^2 y \notin F_4$.
This is a contradiction.

Assume $\ell_4 \sim \ell_1$, that is, $\ell_4 = \nu t_0$ for some $\nu \in \mbC$. 
Since $\msp_{t_1}$ is of type $\frac{1}{5} (1,2,3)$, we have $t_1^2 z \in F_5$.
But since $F_5|_{\Pi} = z \ell_2 \ell_4 + \cdots$, $\ell_2 = t_1$ and $\ell_4 = \nu t_0$, we see $t_1^2 \notin F_5$.
This is a contradiction and the proof is completed.
\end{proof}

We exclude the point $\msp$ assuming the following generality condition.

\begin{Cond} \label{cd:deg12-3}
$y \in a_3$, $\eta - \alpha \delta \ne 0$, $\ell_3 + \beta \ell_2 \not\sim \ell_1$ and $\ell_4 - \delta \ell_1 \not\sim \ell_2$. 
\end{Cond}

Note that $u \in a_6, b_6$ and $v \in b_7,c_7$ since $\msp_u, \msp_v \notin X$.
We set $\Pi = (x = u = v = 0)$.
Then we can write
\[
M|_{\Pi} =
\begin{pmatrix}
0 & y & \alpha z & \ell_1 & \beta y^2 \\
& 0 & \ell_2 & \gamma y^2 & \delta z y \\
& & 0 & \varepsilon z y & y \ell_3 + \eta z^2 \\
& & & 0 & z \ell_4 + \zeta y^3 \\
& & & & 0
\end{pmatrix}.
\]
We see that the coefficients of $z y^2$ and $y^4$ in $F_1|_{\Pi}$ and $F_3|_{\Pi}$ are $\varepsilon - \alpha \gamma$ and $\zeta + \beta \gamma$ respectively and both of them are zero by our choice of coordinates.
By eliminating $\varepsilon = \alpha \gamma$ and $\zeta = - \beta \gamma$, we have
\[
\begin{split}
F_1|_{\Pi} &= \ell_1 \ell_2, \\
F_2|_{\Pi} &= y^2 (\ell_3 + \beta \ell_2) + (\eta - \alpha \delta) z^2 y, \\
F_3|_{\Pi} &= (\ell_4 - \delta \ell_1) z y, \\
F_4|_{\Pi} &= - y \ell_1 \ell_3 + z^2 (\alpha \ell_4 - \eta \ell_1), \\
F_5|_{\Pi} &= -\gamma (\ell_3 + \beta \ell_2) y^3 - \gamma (\eta - \alpha \delta) z^2 y^2 + z \ell_2 \ell_4.
\end{split} 
\]

%\begin{Lem} \label{lem:deg12-3cd}
%If $X$ satisfies \emph{Condition \ref{cd:deg12-3}}, then $\eta - \alpha \delta \ne 0$, $\ell_3 + \beta \ell_2 \not\sim \ell_1$ and $\ell_4 - \delta \ell_1 \not\sim \ell_2$.
%\end{Lem}
%
%\begin{proof}
%If $\eta - \alpha \delta = 0$, then $X$ contains the WCI curve $(x = t_0 = t_1 = u = v = 0)$ of type $(1,5,5,6,7)$, which is impossible by Condition \ref{cd:deg12-3}.
%If $\ell_3 + \beta \ell_2 \sim \ell_1$, then $X$ contains the WCI curve $(x = z = \ell_1 = u = v = 0)$ of type $(1,4,5,6,7)$ which is again impossible by Condition \ref{cd:deg12-3} and the proof is complete.  
%Assume $\ell_4 - \delta \ell_1 \sim \ell_2$ and set $\Pi' = \Pi \cap (\ell_2 = 0)$.
%Then $F_1|_{\Pi'} = F_3|_{\Pi'} = 0$.
%We have $F_2|_{\Pi'} = y f$, where $f = y \ell_3 + (\eta - \alpha \delta) z^2$.
%Since $\ell_4 - \delta \ell_1 = 0$ on $\Pi'$, we have
%\[
%\begin{split}
%F_4|_{\Pi'} &= - y \ell_1 \ell_3 + z^2 (\alpha \ell_4 - \eta \ell_1)
%= - y \ell_1 + (\alpha \delta - \eta) z^2 \ell_1 = - \ell_1 f, \\
%F_5|_{\Pi'} &= - \gamma y F_2|_{\Pi'} = - \gamma y^2 f.
%\end{split}
%\] 
%It follows that $X$ contains the WCI curve $(x = \ell_2 = u = v = f = 0)$ of type $(1,5,6,7,8)$.
%This is impossible by Condition \ref{cd:deg12-3} abd the proof is completed.
%\end{proof}

\begin{Lem}
No singular point of type $\frac{1}{3} (1,1,2)$ is a maximal centre.
\end{Lem}

\begin{proof}
We will show that $\{x,u,v\}$ isolates $\msp$.
It is enough to show that $X \cap \Pi^{\circ}$ is a finite set of points, where $\Pi^{\circ} = \Pi \cap (y \ne 0)$.
We have $F_5|_{\Pi} + \gamma y F_2|_{\Pi} = z \ell_2 \ell_4$.
Since $F_1|_{\Pi} = \ell_1 \ell_2$.
\[
X \cap \Pi^{\circ} = (\ell_1 \ell_2 = F_2|_{\Pi} = F_3|_{\Pi} = F_4|_{\Pi} = z \ell_2 \ell_4 = 0) \cap \Pi^{\circ} 
= \Sigma_1 \cup \Sigma_2,
\]
where
\[
\begin{split}
\Sigma_1 &= (\ell_1 = F_2|_{\Pi} = F_3|_{\Pi} = F_4|_{\Pi} = z \ell_2 \ell_4 = 0) \cap \Pi^{\circ}, \\
\Sigma_2 &= (\ell_2 = F_2|_{\Pi} = F_3|_{\Pi} = F_4|_{\Pi} = 0) \cap \Pi^{\circ}.
\end{split}
\]

Since $\ell_1 \not\sim \ell_2$ and $\ell_1 \not\sim \ell_4$, $\ell_1 = \ell_2 = 0$ and $\ell_1 = \ell_4 = 0$ both imply $t_0 = t_1 = 0$.
Hence we have $(\ell_1 = z \ell_2 \ell_4 = 0) = (t_0 = t_1 = 0) \cup (\ell_1 = z = 0)$ and 
\[
\Sigma_1 = \left((t_0 = t_1 = (\eta - \alpha \delta) z^2 y = 0) \cap \Pi^{\circ} \right) \cup \left((\ell_1 = z = \ell_3 + \beta \ell_2 = 0) \cap \Pi^{\circ}\right)
= \{\msp\}
\]
since $\eta - \alpha \delta \ne 0$ and $\ell_3 + \beta \ell_2 \not\sim \ell_1$ by Condition \ref{cd:deg12-3}.

Since $F_3|_{\Pi} = (\ell_4 - \delta \ell_1) z y$ and $\ell_4 - \delta \ell_1 \not\sim \ell_2$ by Condition \ref{cd:deg12-3}, we have $(\ell_2 = F_3|_{\Pi} = 0) \cap \Pi^{\circ} = (t_0 = t_1 = 0) \cup (\ell_2 = z = 0)$.
Hence
\[
\Sigma_2 = \left((t_0 = t_1 = (\eta - \alpha \delta) z^2 y = 0) \cap \Pi^{\circ} \right) \cup \left((\ell_2 = z = y^2 \ell_3 = - y \ell_1 \ell_3 = 0) \cap \Pi^{\circ} \right) 
= \{\msp\}
\]
since $\ell_3 \not\sim \ell_2$.
Thus, $\{x,u,v\}$ isolates $\msp$.

We see that $y^3 x, y^2 z, y v$ are the monomials of degree $10$ having initial weight $\frac{1}{4}$ and we have $y^3 x, y^2 z \notin F_1$ by our choice of coordinates.
Hence we have $\ord_E (x,u,v) \ge \frac{1}{3} (1,3,4)$ and $L = 6 \varphi^*A -\frac{3}{3} E$ is nef by Lemma \ref{lem:isolnef}.
We compute
\[
(L \cdot B^2) = 6 (A^3) - \frac{3}{3^3} (E^3) = \frac{1}{2} - \frac{1}{2} = 0.
\]
Therefore $\msp$ is not a maximal centre by Lemma \ref{lem:excltc}.
\end{proof}

\subsection{The $\frac{1}{5} (1,1,4)$ point and birational involution}

Let $\msp \in X$ be the point of type $\frac{1}{5} (1,1,4)$.
We assume $\msp = \msp_{t_1}$ after replacing $t_0$ and $t_1$.
Then we have $t_1 t_0 \in F_1$, $t_1 v \in F_3$ and $t_1^2 y \in F_4$ since $\msp$ is of type $\frac{1}{5} (1,1,4)$.
We have $u \in a_6, b_6$ and $v \in b_7, c_7$ since $\msp_u, \msp_v \notin X$.
We see that $\varphi$ is the weighted blowup of $X$ at $\msp$ with weight $\wt (x,z,u) = \frac{1}{5} (1,4,1)$ and it is realized as the embedded weighted blowup with the initial weight $\wt (x,y,z,t_0,u,v) = \iniw = \frac{1}{5} (1,3,4,5,1,2)$.

Let $\pi \colon X \ratmap \mbP := \mbP (1,3,4,5)$ be the projection to the coordinates $x,u,z,t_0$ and let $\pi_Y \colon Y \ratmap \mbP$ the induced rational map.
We take $H \in |\mcO_{\mbP} (1)|$.

\begin{Lem}
The map $\pi_Y$ is a surjective generically finite morphism of degree $2$ such that $B = \pi_Y^*H$.
\end{Lem}

\begin{proof}
We will show that $\pi_Y$ is everywhere defined.
We have an isomorphism
\[
E \cong (F^{\iniw}_1 = F^{\iniw}_3 = F^{\iniw}_4 = 0) \subset \mbP (1_x,3_y,4_z,5_{t_0},1_u,2_v)
\]
and it is enough to show $(x = y = z = t_0 = 0) \cap E = \emptyset$.
We can write $F^{\iniw}_1 = t_0 + g_1$, $F^{\iniw}_3 = v + \alpha u^2 + g_3$ and $F^{\iniw}_4 = y + \beta v u + g_4$, where $g_i \in (x,y,z,t_0)$ and $\alpha, \beta \in \mbC \setminus \{0\}$.
It is now clear that $(x = y = z = t_0 = 0) \cap E = \emptyset$.
This shows that $\pi_Y$ is a morphism.
We have $B = \pi_Y^*H$ since the section $x$ lifts to an anticanonical section on $Y$.
We have $(H^3) = 1/60$ and
\[
(B^3) = (A^3) - \frac{1}{5^3} (E^3) = \frac{1}{12} - \frac{1}{20} = \frac{1}{30},
\]
which shows that $\pi_Y$ is surjective and is generically finite of degree $2$.
\end{proof}

By the same argument as in the proof of Proposition \ref{prop:deg20-birinv}, this lemma implies the following.

\begin{Prop} \label{prop:deg12-birinv}
One of the following holds.
\begin{enumerate}
\item $\msp$ is not a maximal centre.
\item There is a birational involution $\sigma \colon X \ratmap X$ which is a Sarkisov link centred at $\msp$.
\end{enumerate}
\end{Prop}

\subsection{The $\frac{1}{5} (1,2,3)$ point and birational non-rigidity}
\label{sec:deg12nonbr}

Let $\msp$ be the point of type $\frac{1}{5} (1,2,3)$.
We will show that there is a Sarkisov link to a Mori fiber space which is not isomorphic to $X$ starting with the Kawamata blowup $\varphi$.
We denote by $\msq \in X$ the unique singular point of type $\frac{1}{5} (1,1,4)$.

\begin{Lem} \label{lem:deg12defeqlink}
By choice of coordinates, we can assume $\msp = \msp_{t_1}$, $\msq = \msp_{t_0}$ and defining polynomials of $X$ are of the forms:
\[
\begin{split}
F_1 &= t_1 t_0 + v a_3 + u a_4 + f_{10}, \\
F_2 &= t_1 u + v b_4 + u b_5 + g_{11}, \\
F_3 &= t_0 v + v c_5 + \alpha u^2 + u h_6 + h_{12}, \\
F_4 &= t_0^2 y + t_0 (v d_1 + u d_2 + h_8) - \beta u v + v h'_6 + u h_7 + h_{12}, \\
F_5 &= t_1^2 z + t_1 (v e_2 + u e_3 + g_9) + \beta v^2 + v u e_1 + v g_7 + u^2 e_2 + u g_8 + g_{14},
\end{split}
\] 
for some $\alpha, \beta \in \mbC \setminus \{0\}$, $a_i,b_i,\dots, f_i \in \mbC [x,y,z]$, $g_i \in \mbC [x,y,z,t_0]$ and  $h_i, h'_6 \in \mbC [x,y,z,t_1]$ with $t_1^2 y \notin h_{12}$, $t_0^2 z \notin g_{14}$.
Moreover, if $X$ is general, then \emph{Condition \ref{cd:deg12-5link}} below is satisfied.
%$y \in a_3$, $z \in b_4$ and the set
%\[
%(a_3 = a_4 = z + v e_2 + \beta v^2 = 0) \subset \mbP (1_x,3_y,4_z,2_v)
%\]
%consists of distinct $2$ points.
\end{Lem}

\begin{proof}
The syzygy matrix can be written as
\[
M =
\begin{pmatrix}
a_3 & a_4  & A_5 & A_6 \\
0 & B_5 & \alpha u + t_0 b_1 + t_1 b'_1 + b_6 & B_7 \\
0 & 0 & -\beta v + u c_1 + t_0 c_2 + t_1 c'_2 + c_7 & v c'_1 + u c''_2 + t_0 c_3 + t_1 c'_3 + c_8 \\
0 & 0 & 0 & v d_2 + u d_3 + t_0 d_4 + t_1 d'_4 + d_9
\end{pmatrix},
\]
where $\alpha,\beta \in \mbC$, $a_i,b_i,b'_i,c_i,c'_i,c''_i,d_i,d'_i \in \mbC [x,y,z]$ and $A_i,B_i \in \mbC [x,y,z,t_0,t_1,u,v]$.
We will choose suitable coordinates so that the defining polynomials of $X$ are in the desired forms.
First, we choose $t_0$ and $t_1$ so that 
\[
A_5 = t_0 + a_4 b'_1 - a_3 c'_2, \ 
B_5 = t_1 + a_4 b_1 - a_3 c_2.
\]
Then $t_1 t_0$ is the unique monomial in $F_1$ that involves only on $t_0$ and $t_1$ so that $\msp_{t_0}$ and $\msp_{t_1}$ are the $\frac{1}{5} (1,1,4)$ and $\frac{1}{5} (1,2,3)$ points.
We are going to arrange the coordinates so that $\msp_{t_0}$ and $\msp_{t_1}$ are of type $\frac{1}{5} (1,1,4)$ and $\frac{1}{5} (1,2,3)$ respectively. 
Since $\msp_u, \msp_v \notin X$, we have $u \in A_6$, $v \in B_7$ and $\alpha, \beta \ne 0$.
It follows that we can choose $u$ and $v$ so that
\[
A_6 = u - a_3 c'_3, \ 
B_7 = - v + u b_1 + a_3 (d_4 - b_1 c'_3).
\]
By quasi-smoothness of $X$ at $\msp_{t_0}$ (resp.\ $\msp_{t_1}$), we have $t_0^2 y \in F_4$ (resp.\ $t_0^2 z \in F_5$), which implies $y \in c_3$ (resp.\ $z \in d'_4)$.
Hence we can choose $y$ and $z$ so that $c_3 = - y$ and $d'_4 = z + b'_1 c'_3$.
Under the above choice of coordinates, the polynomials $F_1,\dots,F_5$ are in the desired forms.

We have
\[
\begin{split}
F_1 &= t_1 t_0 + v ( -\beta a_3) + (\text{other terms}), \\
F_2 &= t_1 u + v (a_3 c'_1 + a_4) + (\text{other terms}), \\
F_5 &= t_1^2 z + t_1 v (d_2 - b'_1 c'_1 - c'_2) + \beta v^2 + (\text{other terms}).
\end{split}
\]
Clearly $y \in -\beta a_3$ and $z \in a_3 c'_1 + a_4$ for a general $X$ since $\beta \ne 0$.
We see that the set
\[
(-\beta a_3 = a_3 c'_1 + a_4 = z + v (d_2 - b'_1 c'_1 - c'_2) + \beta v^2 = 0)
\]
consists of $2$ distinct points for a general $X$, and the proof is completed.
\end{proof}

\begin{Rem}
Under the above choice of coordinates, $\msp_{t_1}$ is of type $\frac{1}{5} (1_x,2_v,3_y)$ and $\msp_{t_0}$ is of type $\frac{1}{5} (1_x,1_u,4_z)$.
\end{Rem}

We assume the following condition which is satisfied for a general $X$ by the above lemma.

\begin{Cond} \label{cd:deg12-5link}
We have $y \in a_3$ and, under the above choice of coordinates, the set 
\[
(a_3 = b_4 = z + v e_2 + \beta v^2 = 0) \subset \mbP (1_x,3_y,4_z,2_v)
\]
consists of distinct $2$ points.
\end{Cond}

We see that each monomial in $F_2 = t_1 u + v b_4 + u b_5 + g_{11}$ has initial weight at least $6/5$ except for $t_1 u$, so that the weight $\wt (x,y,z,t_0,u,v) = \frac{1}{5} (1,3,4,5,6,2) =: \mbfw$ satisfies the KBL condition.
It follows that $\varphi$ is realized as the embedded weighted blowup with weight $\mbfw$ and we have an isomorphism 
\[
E \cong (t_0 + v a_3 = u + v b_4 = z + v e_2 + \beta v^2 = 0) \subset \mbP,
\]
where $\mbP = \mbP (1_x,3_y,4_z,5_{t_0},6_u,2_v)$.
Let $X \ratmap \mbP (1,3,4,5,6)$ be the projection to $x,y,z,t_0,u$ which is defined outside $\msp$, and denote by $Z$ its image.
Let $\rho \colon Y \ratmap Z$ be the induced birational map.

\begin{Lem} \label{lem:deg12birrho}
$\rho$ is a birational morphism and it is the anticanonical model of $Y$.
\end{Lem}

\begin{proof}
We see that the sections $x,y,z,t_0,u$ lift to plurianticanonical sections on $Y$ and they restrict to $E$ the coordinates $x,y,z,t,u$ of $\mbP$.
It is straightforward to see
\[
(x = y = z = t = u = 0) \cap E = \emptyset
\]
and this implies that $\rho$ is everywhere defined.
For a general point of $Z$, its inverse image via $\rho$ is a single point since we can solve $t_1$ and $v$ in terms of $F_1 = F_2 = 0$ which can be expressed as
\[
\begin{pmatrix}
t_0 & a_3 \\
u & b_4
\end{pmatrix}
\begin{pmatrix}
t_1 \\
v
\end{pmatrix}
= 
- \begin{pmatrix}
u a_4 + f_{10} \\
u b_5 + g_{11}
\end{pmatrix}.
\]
This shows that $\rho$ is birational and thus it is the anticanonical model of $Y$.
\end{proof}

The following lemma will be used in order to show that $\rho$ is a small contraction.

\begin{Lem} \label{lem:divcontmob}
Let $V$ be a $\mbQ$-Fano variety of Picard number one and let $\varphi \colon W \to V$ a $K_W$-negative extremal divisorial contraction with exceptional divisor $E$.
Suppose that $W$ admits a $K_W$-trivial divisorial contraction $\psi \colon W \to U$ which contracts a divisor $G$.
If a prime divisor $D$ on $W$ is $\mbQ$-linearly equivalent to $- \lambda K_W - \mu E$ for some $\lambda,\mu$ with $\mu > 0$, then $D = G$. 
\end{Lem}

\begin{proof}
Note that $\Pic (V) \otimes \mbQ$ is generated by $-K_W$ and $E$, and the cone of effective divisors on $W$ is generated by $E$ and $G$.

Since $\psi \colon W \to U$ is divisorial and $-K_W$-trivial, there are infinitely many curves on $W$ contracted by $\psi$ and they intersect $-K_W$ trivially and $E$ positively.
By \cite[Lemma 2.20]{OkadaII} (see also \cite[]{cheltsov-park}), $\varphi \colon W \to V$ is not a maximal extraction.
This implies that a divisor which is $\mbQ$-linearly equivalent to $- \lambda' K_W - \mu' E$ is not mobile if $\mu' > 0$ (because otherwise $\varphi$ is a maximal extraction).

Let $D \sim_{\mbQ} - \lambda K_W - \mu E$, $\mu > 0$, be a prime divisor.
We assume that $D \ne G$.
Since the cone of effective divisor of $W$ is generated by $E$ and $G$, we can write $D \sim_{\mbQ} k G + l E$ for some rational numbers $k, l > 0$.
Take a positive integer $m$ such that $m D \sim m k G + m l E$ and $m k, m l \in \mbZ$.
This linear equivalence implies that the linear system $|m D|$ is mobile since $D \ne G, E$.
This is a contradiction and the assertion is proved.
\end{proof}

\begin{Lem} \label{lem:deg12flop}
$\rho$ is a flopping contraction.
\end{Lem}

\begin{proof}
We see that the set 
\[
(a_4 = b_4 = 0) \cap E = (a_3 = b_4 = t_0 = u = z + v e_2 + v^2 = 0) \subset \mbP
\]
consists of two points $\{\msq_1, \msq_2\}$ and both of them are mapped to the same point $\msq \in \rho (E)$ via $\rho$, where 
\[
\{\msq\} = (a_3 = b_4 = t_0 = u = 0) \subset \mbP (1,3,4,5,6).
\]
Note that this in particular implies that $\rho$ is not an isomorphism.

It remains to show that $\rho$ is not divisorial.
Assume that $\rho$ is divisorial and let $G$ be the prime divisor on $Y$ contracted by $\rho$.
Since $G$ is contracted by the $B$-trivial contraction $\rho$, we have $(B^2 \cdot G) = 0$.
Since $(B^3) = 1/20$, we compute
\[
0 = (B^2 \cdot G) = k (B^3) - l (B^2 \cdot E) = \frac{1}{20} k - \frac{1}{5^2} l (E^3) = \frac{1}{20} k - \frac{1}{6} l.
\]
Since $k$ and $l$ are integers, we have $G \sim_{\mbQ} m (10 B - 3 E)$ for some positive integer $m$.
We will construct a prime divisor on $Y$ which is $\mbQ$-linearly equivalent to $\lambda B - \mu E$ for some $\lambda,\mu$ with $0 < \lambda < 10$ and $\mu > 0$.
We have
\[
b_4 F_1 - a_3 F_2 = t_1 (t_0 b_4 - u a_3) + b_4 (u a_4 + f_{10}) - a_3 (u b_5 + g_{11}).
\]
Thus, on $X$, we have
\[
 t_1 (t_0 b_4 - u a_3) = - b_4 (u a_4 + f_{10}) + a_3 (u b_5 + g_{11}).
\]
Each monomial in the right-hand side of the above equation vanishes along $E$ to order at least $14/5$.
Let $H \sim_{\mbQ} 9 A$ be the divisor on $X$ defined by $t_0 b_4 - u a_3 = 0$.
We have $\tilde{H} \sim_{\mbQ} 9 \varphi^* A - \frac{14}{5} E = 9 B - E$.  
Note that $\tilde{H}$ is not necessarily irreducible or reduced.
However there is a prime divisor $D$ (which is a component of $\tilde{H}$) such that $D \sim_{\mbQ} \lambda B - \mu E$ such that $\mu > 0$.
The integer $\lambda$ necessarily satisfies $0 < \lambda \le 9$.
This implies that $D \ne G$.
By Lemma \ref{lem:divcontmob}, this is a contradiction and $\rho$ is small.
\end{proof}

Let $\varphi' \colon Y' \to X$ be the Kawamata blowup of $X$ at the $\frac{1}{5} (1,1,4)$ point $\msq = \msp_{t_0}$ with exceptional divisor $E'$.
We see that $\varphi'$ can be realized as the embedded weighted blowup with the initial weight $\wt (x,y,z,t_1,u,v) = \frac{1}{5} (1,3,4,5,1,2)$ so that we have an isomorphism
\[
E' \cong (t_1 + v a_3 + u a_4 = v + \alpha u^2 = y + v d_1 + u d_2 -\beta u v = 0) \subset \mbP'.
\]
where $\mbP' = \mbP (1_x,3_y,4_z,5_{t_1},1_u,2_v)$.

Let $\psi \colon \hat{Y} \to Y$ be the Kawamata blowup of $Y$ at the $\frac{1}{5} (1,1,4)$ point $\varphi^{-1} (\msq)$.
We denote by $\pi \colon X \ratmap \mbP (1,3,4)$ the projection to $x,y,z$ and by $\eta \colon \hat{Y} \ratmap \mbP (1,3,4)$ the induced rational map.
We have the following diagram
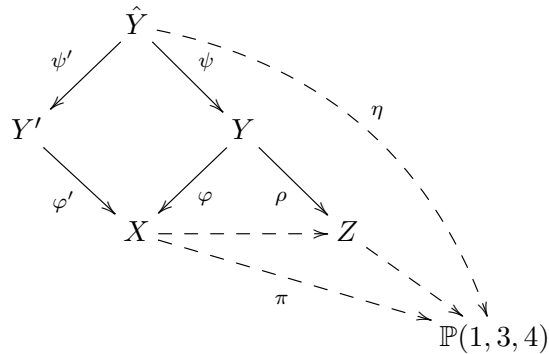
\[
\xymatrix{
& \ar[ld]_{\psi'} \hat{Y} \ar[rd]^{\psi} \ar@{-->}@/^2.3pc/[rrrddd]^{\eta} & & & \\
Y' \ar[rd]_{\varphi'} & & \ar[ld]^{\varphi} Y \ar[rd]_{\rho} & & \\
& X \ar@{-->}[rrrd]_{\pi \hspace{0.6cm}} \ar@{-->}[rr]& & Z \ar@{-->}[rd] & \\
& & & & \mbP (1,3,4)}
\]
where $\psi' \colon \hat{Y} \to Y'$ is the Kawamata blowup of $Y'$ at the $\frac{1}{5} (1,2,3)$ point ${\varphi'}^{-1} (\msp)$ and $\eta$ is the rational map induced by $\pi$.
Note that the exceptional divisors of $\psi$ and $\psi'$ are $\hat{E}'$ and $\hat{E}$ which are the proper transforms of $E'$ and $E$ respectively, where we recall that $E'$ is the exceptional divisor of the Kawamata blowup $\varphi' \colon Y' \to X$ at the $\frac{1}{5} (1,1,4)$ point $\msq = \msp_{t_0}$.
We set $B = -K_Y$ and $\hat{B} = -K_{\hat{Y}}$.
It is straightforward to compute that $(B^3) = 1/20$, $(\hat{B}^3) = 0$.

\begin{Lem}
$\eta$ is a morphism which is an elliptic fibration.
Moreover, $\hat{E}$ and $\hat{E}'$ are respectively $2$-section and $3$-section of $\eta$.
\end{Lem}

\begin{proof}
The indeterminacy locus of the projection $\pi \colon X \ratmap \mbP (1,3,4)$ is the set $\Xi := (x = y = z = 0) \cap X$.
We have 
\[
F_1 (0,0,0,t_0,t_1,u,v) = t_1 t_0, \ 
F_2 (0,0,0,t_0,t_1,u,v) = t_1 u.
\]
so that $\Xi = \Xi_1 \cup \Xi_2$, where 
\[
\Xi_1 = (x = y = z = t_1 = 0) \cap X, \ 
\Xi_2 = (x = y = z = t_0 = u = 0) \cap X.
\]
By looking at the other polynomials $F_3, F_4, F_5$, it is easy to check that $\Xi_1 = \{\msp_{t_0}\}$ and $\Xi_2 = \{\msp_{t_1}\}$.
This shows that $\pi$ is defined outside $\{\msp_{t_0}, \msp_{t_1}\}$.
The proper transforms of the sections $x,y,z$ on $Y$ restricts to the coordinates $x,y,z$ on $E \subset \mbP$ and we have $(x = y = z = 0) \cap E = \emptyset$.
This shows that $\eta$ is defined at every point of $\hat{E}$.
For $\lambda,\mu \in \mbC$, we set $S_{\lambda} = (y - \lambda x^3 = 0) \cap X$ and $T_{\mu} = (z - \mu x^4 = 0) \cap X$.
We see that $\tilde{S}_{\lambda} \cap \tilde{T}_{\mu}$ is the fiber of $\pi \circ \varphi \colon Y \ratmap \mbP (1,3,4)$ over the point $(1\!:\!\lambda\!:\!\mu)$, $\tilde{S}_{\lambda}|_E$ and $\tilde{T}_{\mu}|_E$ are hyperplane sections of degree $3$ and $4$ on $E \subset \mbP$, so that we have
\[
(\tilde{S}_{\lambda} \cdot \tilde{T}_{\mu} \cdot E) = (\tilde{S}_{\lambda}|_E \cdot \tilde{T}_{\mu} |_E)_E = \frac{3 \cdot 4 \cdot 5 \cdot 6 \cdot 4}{1 \cdot 3 \cdot 4 \cdot 5 \cdot 6 \cdot 2} = 2.
\]
This shows that $\hat{E}$ is a $2$-section of $\eta$.
Here we explain the above computation in more detail.
Since $E$ is a complete intersection in $\mbP$ defined by equations of degree $5,6,4$ and $\tilde{S}_{\lambda}|_E$, $\tilde{T}_{\mu}|_E$ correspond to hypersurfaces in $\mbP$ of degree $3$, $4$ respectively, we have $(\tilde{S}_{\lambda}|_E \cdot \tilde{T}_{\mu}|_E)_E = 3 \cdot 4 \cdot 5 \cdot 6 \cdot 4 \cdot (\mcO_{\mbP} (1))^5$.

The proper transforms of the sections $x,y,z$ on $Y$ restricts to the coordinates $x,y,z$ on $E' \subset \mbP'$ and we have $(x = y = z = 0) \cap E' = \emptyset$.
This shows that $\eta$ is defined at every point of $\hat{E}'$.
We see that $S'_{\lambda} \cap T'_{\lambda}$, where $S'_{\lambda} = {\psi'}_*^{-1} S_{\lambda}$ and $T'_{\mu} = {\psi'}_*^{-1} T_{\mu}$, is the fiber of $\varphi' \circ \pi$ over the point $(1\!:\!\lambda\!:\!\mu)$, $S'_{\lambda}|_{E'}$ and $T'_{\mu}|_{E'}$ are hyperplane sections of degree $3$ and $4$ on $F \subset \mbP'$, so that we have
\[
(S'_{\lambda} \cdot T'_{\mu} \cdot E') = (S'_{\lambda}|_{E'} \cdot T'_{\mu} |_{E'})_{E'} = \frac{3 \cdot 4 \cdot 5 \cdot 2 \cdot 3}{1 \cdot 3 \cdot 4 \cdot 5 \cdot 1 \cdot 2} = 3.
\]
This shows that $\hat{E}'$ is a $3$-section of $\eta$.
We note that the intersections $\tilde{S}_{\lambda} \cap \tilde{T}_{\mu} \cap E$ and $S'_{\lambda} \cap T'_{\mu} \cap E$ can also be computed explicitly using local coordinates.)

Thus $\hat{\pi}$ is everywhere defined.
It is clear that the sections $x,y,z$ lift to sections of $\hat{B}, 3 \hat{B}, 4 \hat{B}$ respectively, so that $\eta$ is the anticanonical morphism and it is an elliptic fibration.
\end{proof}

By Lemma \ref{lem:deg12flop}, $\rho$ is a flopping contraction.
Let $\tau = \tau_0 \colon Y \ratmap Y_1$ be the flop of $\rho$.
Then $Y_1$ admits a $K_{Y_1}$-negative extremal ray because otherwise $K_{Y_1}$ is nef and big which is impossible.
There are three options: $Y_1$ is a Mori fiber space, $Y_1$ admits a $K_{Y_1}$-negative divisorial contractions to a $\mbQ$-Fano 3-fold or $Y_1$ admits a flip $Y_1 \ratmap Y_2$.
In the last case, $Y_2$ also have the same options since $K_{Y_2}$ is not nef and big.
Thus the flop $Y \ratmap Y_1$ followed by a sequence of flips gives a $2$-ray game which ends with a Mori fiber space, that is, we have a Sarkisov link $\sigma \colon X \ratmap \bar{X}/\bar{S}$ to a Mori fiber space.
We will show that $\bar{X}$ is not isomorphic to $X$, which requires all the results of this section.

\begin{Thm} \label{thm:deg12nonrig}
The Sarkisov link $\sigma$ starting with the Kawamata blowup of $X$ at the $\frac{1}{5} (1,2,3)$ point $\msp$ is a link to a Mori fiber space which is not isomorphic to $X$.
In particular, $X$ is not birationally rigid.
\end{Thm}

\begin{proof}
We assume $\bar{X} \cong X$.
Then the link $\sigma$ sits in the diagram:
\[
\xymatrix{
Y \ar[d]_{\varphi} \ar@{-->}[r]^{\tau = \tau_0} & Y_1 \ar@{-->}[r]^{\tau_1} & \cdots \ar@{-->}^{\tau_{m-1}}[r] & Y_m \ar@{-->}[r]^{\tau_m} & \bar{Y} \ar[d]_{\bar{\varphi}}\\
X & & & & X}
\] 
where $\tau_i$ is a flip for $i \ge 1$ and $\bar{\varphi}$ is an extremal divisorial contraction. 
We see that $\bar{\varphi}$ coincides with either $\varphi$ or $\varphi'$ because a centre other than $\msp$ and $\msp_{t_0}$ is not a maximal centre.
By Proposition \ref{prop:deg12-birinv} (see also \cite[Lemma 3.2]{OkadaII}) , the Sarkisov link starting with $\varphi'$ ends with $\varphi'$.
By the uniqueness of $2$-ray game starting with a given divisorial extraction, $\bar{\varphi}$ cannot be $\varphi'$ and hence $\bar{\varphi} = \varphi$.
Now $\bar{Y} \cong Y$ so that it does not admit an inverse flip, which implies that $\tau_m$ cannot be a flip.
Thus $m = 0$, that is, the link involves only the flop $\tau$. 

We have the following diagram
\[
\xymatrix{
Y \ar@{-->}[rr]^{\tau} \ar[d]_{\varphi} \ar[rd]^{\rho} & & \ar[ld]_{\rho'} Y \ar[d]^{\varphi} \\
X & Z & X}
\]
where $\rho'$ is a flopping contraction.
Note that $\rho'$ can be decomposed as $\rho' = \theta \circ \rho$, where $\theta \colon Z \to Z$ is an automorphism, since $\tau$ induces an isomorphism between the anticanonical model $Z$ of $Y$.
Let $\hat{\tau} \colon \hat{Y} \ratmap \hat{Y}$ be the birational automorphism induced by $\tau$.
We set $N = \hat{B} + \varepsilon \hat{E}'$ for $0 < \varepsilon < \frac{1}{5}$, which is nef and big since $\psi^*B = \hat{B} + \frac{1}{5} \hat{E}'$ is nef and big, and $\hat{B}$ is nef.
We choose $0 < \varepsilon \ll 1/5$ so that $N$ is $\psi$-ample.
Let $\hat{\rho} \colon \hat{Y} \to \hat{Z}$ be the contraction associated with $N$.

We will show that the curves contracted by $\hat{\rho}$ are precisely the proper transforms of the flopping curves on $Y$.
Let $\Gamma \subset Y$ be a flopping curve.
Then
\[
0 \le (\hat{B} \cdot \hat{\Gamma}) = (B \cdot \Gamma) - \frac{1}{5} (\hat{E}' \cdot \hat{\Gamma}) = - \frac{1}{5} (\hat{E}' \cdot \hat{\Gamma}) \le 0.
\]
This shows $\hat{\Gamma} \cap \hat{E}' = \emptyset$ and $(\hat{B} \cdot \hat{\Gamma}) = 0$.
In particular, $\hat{\Gamma}$ is contracted by $\hat{\rho}$.
Let $\Delta \subset \hat{Y}$ be an irreducible curve on $\hat{Y}$ which is contracted by $\hat{\rho}$.
Note that $\Delta \not\subset \hat{E}'$ since $N$ is $\psi$-ample.
Then
\[
0 = (N \cdot \Delta) = (\hat{B} \cdot \Delta) + \varepsilon (\hat{E}' \cdot \Delta) \ge \varepsilon (\hat{E}' \cdot \Delta) \ge 0,
\]
which implies $\Delta \cap \hat{E}' = \emptyset$ and $0 = (\hat{B} \cdot \Delta) = (B \cdot \psi_*\Delta)$.
Thus $\Delta$ is the proper transform of a flopping curve on $Y$.

By the above argument, the curves contracted by $\hat{\rho}$ form a $K_{\hat{Y}}$-trivial extremal ray and $\hat{\rho}$ is a flopping contraction over $\mbP (1,3,4)$.
Moreover $\hat{Z}$ is obtained as the Kawamata blowup of $Z$ at the $\frac{1}{5} (1,1,4)$ point $\bar{\msq} := \rho (\varphi^{-1} (\msq))$.
Since the point $\varphi^{-1} (\msq) \in Y$ is the unique singular point of $\frac{1}{5} (1,1,4)$, the point $\bar{\msq} \in Z$ is the unique point of type $\frac{1}{5} (1,1,4)$.
Hence $\theta$ fixes $\bar{\msq}$.
It follows that the birational map $\hat{\tau} \colon \hat{Y} \ratmap \hat{Y}$ is the flop of $\hat{\rho}$ and we have the following commutative diagram
\[
\xymatrix{
\hat{Y} \ar@{-->}[rr]^{\hat{\tau}} \ar[rd]_{\eta} & & \ar[ld]^{\eta'} \hat{Y} \\
& \mbP (1,3,4)}
\]
where $\eta' = \chi \circ \eta$ for some automorphism $\chi$ of $\mbP (1,3,4)$ since the flop $\hat{\tau}$ induces an isomorphism of the anticanonical model $\mbP (1,3,4)$ of $\hat{Y}$.
Thus $\hat{\tau}$ is an isomorphism in codimension $1$ and it induces an isomorphism between the generic fibers of $\eta$ and $\eta'$.

We have $\hat{\tau}_* \hat{B} = \hat{B}$ since $\hat{\tau}$ is small.
By construction, we have $\hat{\tau}_*\hat{E}' = \hat{E}'$ (because $\theta (\bar{\msq}) = \bar{\msq}$) .
Since the Weil divisor class group of $\hat{Y}$ is generated by $\hat{B}$, $\hat{E}$ and $\hat{E}'$, we can write $\hat{\tau}_* \hat{E} = \alpha \hat{B} - \beta \hat{E} + \gamma \hat{E}'$ for some integers $\alpha, \beta, \gamma$.
Clearly $\alpha \ge 0$ since $\hat{\tau}_* \hat{E}$ is effective and non-zero.
Note that $\tau_* E = \alpha B - \beta E$ and since $\tau$ is a flop we have $\beta > 0$.
If $\alpha = 0$, then $\tau_*E = - \beta E$ and this is a contradiction since $\tau_*E$ is effective.
Hence $\alpha > 0$. 
We have
\[
(\hat{\tau}^2)_* \hat{E} = \alpha (1-\beta) \hat{B} + \beta^2 \hat{E} + \gamma (1-\beta) \hat{F}.
\]
Since $(\hat{\tau}^2)_* \hat{E}$ is effective, we have $\alpha (1-\beta) \ge 0$, which implies $\beta \le 1$.
Thus we have $\beta = 1$.
Since $\hat{\tau}$ induces an isomorphism between generic fibers of the elliptic fibrations $\eta$ and $\eta'$, $\hat{\tau}_*\hat{E}$ is a $2$-section of $\eta'$.
Clearly $\hat{E}$ and $\hat{E}'$ are $2$-section and $3$-section respectively.
Then, for a general $\eta'$-fiber $C'$, we have
\[
2 = (\hat{\tau}_* \hat{E} \cdot C') = \alpha (\hat{B} \cdot C') - (\hat{E} \cdot C') + \gamma (\hat{E}' \cdot C') = -2 + 3 \gamma.
\]
This is a contradiction since $\gamma \in \mbZ$.
Therefore, $\sigma$ cannot be a birational automorphism of $X$.
\end{proof}

\begin{Rem}
We are unable to give an explicit construction of the link $\sigma$ and we do not even understand whether the target Mori fiber space $\bar{X}/\bar{S}$ is a strict Mori fiber space or not.
\end{Rem} 

\section{Pfaffian Fano $3$-fold of degree $1/4$} \label{sec:deg4}

Let $X = X_{7,8,8,9,10} \subset \mbP (1_x,2_y,3_{z_0},3_{z_1},4_{t_0},4_{t_1},5_u)$ be a Pfaffian Fano $3$-fold of degree $1/4$.
The main aim of this section is to prove that there is a Sarkisov link centred at the $\frac{1}{4} (1,1,3)$ point to a Mori fiber space other than $X$.
This implies that $X$ is not birationally rigid.
For the rigorous proof, we need to exclude or untwist the other centres, so we will exclude points of type $\frac{1}{2} (1,1,1)$ and construct a Sarkisov link centred at each $\frac{1}{3} (1,1,2)$ point which is a birational involution. 
The syzygy matrix of $X$ and the defining polynomials are given as follows:
\[
M =
\begin{pmatrix}
0 & a_2 & a_3 & a'_3 & a_4 \\
& 0 & b_4 & b'_4 & b_5 \\
& & 0 & c_5 & c_6 \\
& & & 0 & d_6 \\
& & & & 0
\end{pmatrix}
\hspace{1.5cm}
\begin{aligned}
F_1 &= a_2 c_5 - a_3 b'_4 + a'_3 b_4 \\  
F_2 & = a_2 c_6 - a_3 b_5 + a_4 b_4 \\ 
F_3 &= a_2 d_6 - a'_3 b_5 + a_4 b'_4 \\ 
F_4 &= a_3 d_6 - a'_3 c_6 + a_4 c_5 \\
F_5 &= b_4 d_6 - b'_4 c_6 + b_5 c_5
\end{aligned}
\]
The basket of singularities of $X$ is as follows
\[
\left\{ 3 \times \frac{1}{2} (1,1,1), 3 \times \frac{1}{3} (1,1,2), \frac{1}{4} (1,1,3) \right\}.
\]

\subsection{Exclusion of the $\frac{1}{2} (1,1,1)$ points}

Let $\msp$ be a $\frac{1}{2} (1,1,1)$ point.
Throughout the present subsection, we assume $y \in a_2$ and then, re-scaling $y$, we assume that the coefficient of $y$ in $a_2$ is $1$.
Replacing $y,t_0,t_1$, we assume $\msp = \msp_y$.
We have $u^2 \in F_5$ since $\msp_u \notin X$, which implies $u \in b_5, c_5$.
It follows that $y u \in F_1$.
After replacing $u$, we assume that $y u$ is the unique monomial in $F_1$ which is divisible by $y$.

For the entries of the syzygy matrix $M$, we can write $a_3 = \ell_1 + (\text{other terms})$, $a'_3 = \ell_2 + (\text{other terms})$, $b_5 = y \ell_3 + (\text{other terms})$, $c_5 = y \ell_4 + (\text{other terms})$, $c_6 = \delta y^3 + q_1 + (\text{other terms})$ and $d_6 = \varepsilon y^3 + q_2 + (\text{other terms})$, where $\delta, \varepsilon \in \mbC$, $\ell_1,\dots,\ell_4$ and $q_1, q_2$ are respectively linear and quadratic forms in $t_0,t_1$.
Let $\beta, \gamma \in \mbC$ be the coefficients of $y^2$ in $b_4$ and $b'_4$ respectively.
We exclude the point $\msp$ assuming the following generality condition:

\begin{Cond} \label{cd:deg4-2}
We have $y \in a_2$ and the system of equations 
\[
q_1 - \ell_1 \ell_3 = q_2 - \ell_2 \ell_3 = \beta q_2 - \gamma q_1 + \ell_3 \ell_4 = 0
\]
does not have a non-trivial solution.
\end{Cond}

\begin{Lem}
If $X$ satisfies \emph{Condition \ref{cd:deg4-2}}, then no singular point of type $\frac{1}{2} (1,1,1)$ is a maximal centre.
\end{Lem}

\begin{proof}
We will prove that the set $\{x,t_0,t_1,u\}$ isolates $\msp$.
We set $\Pi = (x = t_0 = t_1 = u = 0)$.
Then we can write
\[
M|_{\Pi} =
\begin{pmatrix}
0 & y & \ell_1 & \ell_2 & \alpha y^2 \\
& 0 & \beta y^2 & \gamma y^2 & y \ell_3 \\
& & 0 & y \ell_4 & \delta y^3 + q_1 \\
& & & 0 & \varepsilon y^3 + q_2 \\
& & & & 0
\end{pmatrix},
\]
where $\alpha,\beta,\dots,\varepsilon \in \mbC$ and $\ell_i, q_i$ are polynomials in $z_0,z_1$ which are linear and quadratic respectively.
Hence we have
\[
\begin{split}
F_1|_{\Pi} &= y^2 (\ell_4 - \gamma \ell_1 + \beta \ell_2), \\
F_2|_{\Pi} &= (\delta + \alpha \beta) y^4 + y (q_1 - \ell_1 \ell_3), \\
F_3|_{\Pi} &= (\varepsilon + \alpha \gamma) y^4 + y (q_2 - \ell_2 \ell_3), \\
F_4|_{\Pi} &= y^3 (\varepsilon \ell_1 - \delta \ell_2 + \alpha \ell_4) + \ell_1 q_2 - \ell_2 q_1, \\
F_5|_{\Pi} &= (\beta \varepsilon - \gamma \delta) y^5 + y^2 (\beta q_2 - \gamma q_1 + \ell_3 \ell_4)
\end{split}
\]
By our choice of coordinates, there is no monomial in $F_1$ divisible by $y$ other than $y u$, so that $\ell_4 - \gamma \ell_1 + \beta \ell_2 = 0$.
Since $\msp = \msp_y \in X$, we see that the coefficients of $y^4$, $y^4$ and $y^5$ in $F_2$, $F_3$ and $F_5$ are zero, which implies 
\[
\delta + \alpha \beta = \varepsilon + \alpha \gamma = \beta \varepsilon - \gamma \delta = 0.
\] 
Combining the above observations, we have
\[
\varepsilon \ell_1 - \delta \ell_2 + \alpha \ell_4 
= (\varepsilon + \alpha \gamma) \ell_1 - (\delta + \alpha \beta) \ell_2 = 0,
\]
hence $F_1|_{\Pi} = 0$ and
\[
F_2|_{\Pi} = y (q_1 - \ell_1 \ell_3),
F_3|_{\Pi} = y (q_2 - \ell_2 \ell_3),
F_4|_{\Pi} = \ell_1 q_2 - \ell_2 q_1,
F_5|_{\Pi} = y^2 (\beta q_2 - \gamma q_1 + \ell_3 \ell_4).
\]
By Condition \ref{cd:deg4-2}, $X \cap \Pi$ consists of $\msp$ and the $3$ points of type $\frac{1}{3} (1,1,2)$.
Thus $\{x,t_0,t_1,u\}$ isolates $\msp$.

We see that $y^3 x, y^2 z_0,y^2 z_1$ and $y u$ are the monomials of degree $7$ having initial weight $\frac{1}{2}$.
By our choice of coordinates, $y u$ is the unique monomial with initial weight $\frac{1}{2}$.
It follows that $\ord_E (x,t_0,t_1,u) \ge \frac{1}{2} (1,2,2,3)$.
Hence $L = 4 \varphi^*A - \frac{2}{2} E$ is nef by Lemma \ref{lem:isolnef} and we compute
\[
(L \cdot B^2) = 4 (A^3) - \frac{2}{2^3} (E^3) = \frac{4}{4} - \frac{2}{2} = 0.
\]
Therefore, $\msp$ is not a maximal centre by Lemma \ref{lem:excltc}.
\end{proof}

\subsection{The $\frac{1}{3} (1,1,2)$ points and birational involutions}

Let $\msp \in X$ be a point of type $\frac{1}{3} (1,1,2)$.
For a polynomial $f = f (x,y,z_0,z_1,t_0,t_1,u)$, we denote $\bar{f} = f (0,0,z_0,z_1,t_0,t_1,0)$.
Note that, for the entries $a_3, a'_3$ and $a_4,b_4,b'_4$ of the syzygy matrix of $X$, $\bar{a}_3, \bar{a}'_3$ and $\bar{a}_4, \bar{b}_4, \bar{b}'_4$ are linear forms in $z_0,z_1$ and $t_0,t_1$, respectively.
Note also that $\bar{c}_6$ and $\bar{d}_6$ are quadratic forms in $z_0,z_1$.

\begin{Cond} \label{cd:deg4-3}
The set
\[
(-\bar{a}_3 \bar{b}'_4 + \bar{a}'_3 \bar{b}_4 = \bar{a}_3 \bar{d}_6 - \bar{a}'_3 \bar{c}_6 = \bar{a}_4 = 0) \subset \mbP (3_{z_0}, 3_{z_1}) \times \mbP (4_{t_0}, 4_{t_1})
\]
is empty.
\end{Cond}

It is clear that Condition \ref{cd:deg4-3} is satisfied for a general $X$ and we assume that $X$ satisfies it.

\begin{Rem} \label{rem:modifmat}
Let $X$ be a Paffian Fano 3-fold defined by the syzygy matrix 
\[
M =
\begin{pmatrix}
a_2 & a_3 & a'_3 & a_4 \\
0 & b_4 & b'_4 & b_5 \\
& 0 & c_5 & c_6 \\
& & 0 & d_6 \\
\end{pmatrix}
\]
and let $F_1,\dots,F_5$ be defining polynomials.
For $\alpha \in \mbC$, the matrices
\[
M_{\alpha} =
\begin{pmatrix}
a_2 & a_3 - \alpha a'_3 & a'_3 & a_4 \\
0 & b_4 - \alpha b'_4 & b'_4 & b_5 \\
& 0 & c_5 & c_6 - \alpha d_6 \\
& & 0 & d_6 \\
\end{pmatrix}, \quad
M'_{\alpha} =
\begin{pmatrix}
a_2 & a_3 & a'_3 - \alpha a_3 & a_4 \\
0 & b_4 & b'_4 - \alpha b_4 & b_5 \\
& 0 & c_5 & c_6 \\
& & 0 & d_6 - \alpha c_6 \\
\end{pmatrix}
\]
both define the same Pfaffian 3-fold $X$ with defining polynomials $F_1,F_2-\alpha F_3,F_3,F_4,F_5$ and $F_1,F_2,F_3- \alpha F_2, F_4,F_5$, respectively.
\end{Rem}

The following choice of coordinates will also be used in the next subsection.

\begin{Lem} \label{lem:coord1/4}
Let $\msp \in X$ be a point of type $\frac{1}{3} (1,1,2)$ and $\msq \in X$ the point of type $\frac{1}{4} (1,1,3)$. 
By a choice of coordinates, we can assume that $\msp = \msp_{z_1}$, $\msq = \msp_{t_1}$ and the polynomials $F_1, \dots,F_5$ are written as follows:
\[
\begin{split}
F_1 &= t_1 z_1 + u a_2 + t_0 a_3 + a_7, \\
F_2 &= t_1 t_0 + u b_3 + t_0 b_4 + z_1 b_5 + b_8, \\
F_3 &= z_1 u + u c_3 + t_0^2 + t_0 c_4 + t_1 g_4 + z_1 c_5 + c_8, \\
F_4 &= z_1^2 z_0 + u t_0 + u g'_4 + t_0 g_5 + t_1 g'_5 + z_1 d_6 +  d_9, \\
F_5 &= t_1^2 y + t_1 (u e_1 + h_6) + u^2 + u h_5 + h_{10},
\end{split}
\]
where $a_i,b_i,c_i,d_i,e_i \in \mbC [x,y,z_0]$, $g_i \in \mbC [x,y,z_0,z_1]$ and $h_i \in \mbC [x,y,z_0,z_1,t_0]$ are all contained in the ideal $(x,y,z_0)$ and satisfy $z_0 \in b_3$ and $z_1^3 x \notin h_{10}$.
\end{Lem}

\begin{proof}
We have $u \in b_5, c_5$ since $u^2 \in F_5$ by quasi-smoothness of $X$.
The equations $\bar{a}_4 \bar{b}_4 = \bar{a}_4 \bar{b}'_4 = 0$ has a unique non-trivial solution and it corresponds to the $\frac{1}{4} (1,1,3)$ point of $X$.
It follows that $\bar{b}_4 = \bar{b}'_4 = 0$ has no non-trivial solution and the solution $\bar{a}_4 = 0$ corresponds to the $\frac{1}{4} (1,1,3)$ point.
We choose coordinates so that $\msp = \msp_{z_1}$ and $\msq = \msp_{t_1}$, which are equivalent to $z_0 \mid (\bar{a}_3 \bar{d}_6 - \bar{a}'_3 \bar{c}_6)$ and $\bar{a}_4 = t_0$.
By suitable modifications of the matrix $M$ in Remark \ref{rem:modifmat}, we may assume $\bar{a}_3 = z_0$.
We have $t_1 \in b_4$ because otherwise the set in Condition \ref{cd:deg4-3} contains the point $((0\!:\!1), (0\!:\!1))$ which is impossible.
Again by a suitable modification of $M$, we may assume $\bar{b}'_4 = t_0$.
Then, since neither $\bar{b}_4 = \bar{b}'_4 = 0$ nor $\bar{a}_3 = \bar{a}'_3 = 0$ has nontrivial solution, we have $t_1 \in \bar{b}_4$ and $z_1 \in a'_3$. 
Replacing $t_1 \mapsto t_1 - \varepsilon_1 t_0$ and $z_1 \mapsto z_1 - \varepsilon_2 z_0$ for some $\varepsilon_1, \varepsilon_2 \in \mbC$, we may assume $\bar{b}_4 = t_1, \bar{a}_4 = t_0$ and $\bar{a}'_3 = z_1$.
So far we choose coordinates so that $\msp = \msp_{z_1}, \msq = \msp_{t_1}$, $\bar{a}_3 = z_0, \bar{a}'_3 = z_1, \bar{b}_4 = t_1, \bar{b}'_4 = t_0$ and $z_0 \mid \bar{c}_6$, where the last assertion follows from $z_0 \mid (\bar{a}_3 \bar{d}_6 - \bar{a}'_3 \bar{c}_6)$ and $\bar{a}_3 = z_0$.

We further replace coordinates which preserve the above properties.
We replace $u$ so that $c_5 = u$.
We replace $z_0 \mapsto h_3 (x,y)$ and $z_1 \mapsto z_1 - h'_3 (x,y)$ for suitable $h_3, h'_3 \in \mbC [x,y]$ so that $a_3 = z_0$ and $a'_3 = z_1$.
Now we can write the syzygy matrix $M$ as follows
\[
M =
\begin{pmatrix}
a_2 & z_0 & z_1 & t_0 + A_4 \\
0 & t_1 + B_4 & t_0 + z_1 b'_1 + b'_4 & \alpha u + t_0 e_1 + t_1 e'_1 + B_5 \\
& 0 & u & u c_1 + t_0 c_2 + t_1 c_2' + C_6 \\
& & 0 & u d_1 + t_0 d_2 + t_1 d'_2 + D_6 \\
\end{pmatrix},
\]
where $\alpha \in \mbC \setminus \{0\}$, $a_2,a_3,a'_3,\dots,d'_2,e_1,e'_1 \in \mbC [x,y,z_0]$ and $A_4, B_4, B_5, C_6, D_6 \in \mbC [x,y,z_0,z_1]$.
We replace $t_0 \mapsto t_0 - A_4 - a_2 c'_2 + e'_1 z_0$ so that, after the replacement, we have $A_4 = - a_2 c'_2 + e'_1 z_0$.
We then replace $t_1 \mapsto t_1 + b'_1 z_0 - B_4$ so that, after the replacement, we have $B_4 = b'_1 z_0$.

We claim that $y \in d'_2$.
Indeed, since $\msq = \msp_{t_1}$ is of type $\frac{1}{4} (1,1,3)$, we have $t_1^2 y \in F_5$.
The terms in $F_5$ divisible by $t_1^2$ are computed as $t_1^2 d'_2$.
Hence $y \in d'_2$ and the claim is proved.
We replace $y$ so that $d'_2 = y$.
We finish the choice of coordinates and in the following  we observe that this is the desired choice of coordinates.

We compute $F_1, \dots,F_5$.
In the following descriptions, omitted terms $\cdots$ consist of monomials in variables $x,y,z_0$.
We have 
\[
\begin{split}
F_1 &= t_1 z_1 + u a_2 - t_0 z_0 + \cdots, \\
F_2 &= t_1 t_0 + u (a_2 c_1 - \alpha z_0) + t_0 (a_2 c_2 + b'_1 z_0) + a_2 C_6 - z_0 B_5 + \cdots, \\
F_3 &= - \alpha z_1 u + u a_2 d_1 + t_0^2 + t_0 (- z_1 e_1 + \cdots) + t_1 (a_2 d'_2 - z_1 e'_1) + a_2 D_6 - z_1 B_5 + \cdots, \\
F_4 &= u (t_0 - z_1 c_1 + \cdots) + t_1 (z_0 y - z_1 c'_2) + t_0 (z_0 d_2 - z_1 c_2) + z_0 D_6 - z_1 C_6.
%F_5 &= t_1^2 + t_1 (u (d_1 + e'_1) + t_0 (d_2 - c'_2) + D_6 - c'_2 (z_1 b'_1 + b'_4) + \cdots) + \alpha u^2 + u (t_0 (e_1 - c_1) - c_1 b'_1 z_1 + B_5 + \cdots) - t_0^2 c_2 + t_0 (- b'_1 c_2 z_1 + C_6 + \cdots) + 
\end{split}
\]
Recall that $z_0 \bar{D}_6 - z_1 \bar{C}_6 = 0$ has three distinct solutions (corresponding to three points of type $\frac{1}{3} (1,1,2)$) and, by our choice of coordinates, $z_0 \mid z_0 \bar{D}_6 - z_1 \bar{C}_6$.
It follows that $z_1^2 \notin C_6$ and $z_1^2 \in D_6$.
Thus it is easy to see that $F_1, F_2, F_3$ are in the form described in the statement after rescaling $u$.
We have $z_1^2 z_0 \in F_4 = z_0 D_6 - z_1 C_6$, which shows that $F_4$ is also in the desired form.
Although we do not write down $F_5$ explicitly here, it is easy to verify that 
\[
F_5 = t_1^2 y + t (u e_1 + h_6) + \beta u^2 + u h_5 + h_{10}
\] 
for some $\beta \in \mbC \setminus \{0\}$, $e_i \in \mbC [x,y,z_0]$ and $h_i \in \mbC [x,y,z_0,z_1,t_0]$.
It is easy to observe that $h_5, h_{10} \in (x,y,z_0)$ because there degree is not divisible by $3$ and it cannot contain a power of $z_1$.
This also explains that $g_i, g'_i \in (x,y,z_0)$.
Note that $h_6 = D_6 - b'_1 c'_2 z_1 + \cdots$ and it contains $z_1^2$.
By replacing $F_5$ by $F_5 - \gamma z_1 F_1$, we can eliminate the term $z_1^2$ in $h_6$.
Finally, replacing $y \mapsto \beta y$ and then replacing $F_5$ by $\frac{1}{\beta} F_5$, we may assume that $\beta = 1$.
This completes the proof.
\end{proof}  

We choose and fix coordinates as above.
It is easy to see that $z_1 t_1$ is the unique monomial in $F_1 = z_1 t_1 + u a_2 + t_0 a_3 + a_7$ having initial weight $\frac{1}{3}$ since $a_i = a_i (x,y,z_0)$ has initial weight $\frac{i}{3}$.
The Kawamata blowup $\varphi \colon Y \to X$ at $\msp$ is realized as the embedded weighted blowup at $\msq$ with weight $\wt (x,y,z_0,t_0,t_1,u) = \frac{1}{3} (1,2,3,1,4,2) =: \mbfw$.

Let $\pi \colon X \ratmap \mbP := \mbP (1,2,3,4)$ be the projection to the coordinates $x,y,z_0,t_1$ and let $\pi_Y \colon Y \ratmap \mbP (1,2,3,4)$ the induced rational map.
We take $H \in |\mcO_{\mbP} (1)|$.

\begin{Lem}
The map $\pi_Y$ is a surjective generically finite morphism of degree $2$ such that $B = \pi_Y^*H$.
\end{Lem}

\begin{proof}
By Lemma \ref{lem:coord1/4}, it s easy to observe that the indeterminacy locus of $\pi$, which is the set $(x = y = z_0 = t_1 = 0) \cap X$, consists of the single point $\msp$ since $a_i,\dots,e_i,g_i,g'_i,h_i$ all vanish along $(x = y = z_0 = 0)$.
We have an isomorphism
\[
E \cong (t_1 + u a_2 + t_0 a_3 = u + \alpha t_0^2 + \gamma t_0 x = z_0 + u t_0 + \delta u x = 0) \subset \mbP (1_x,2_y,3_{z_0},1_{t_0},4_{t_1},2_u),
\]
where $\gamma, \delta \in \mbC$ are the coefficients of $t_0 z_1 x, z_1 x$ in $h_8, g_4$, respectively.
The sections $x,y,z_0,t_1$ lift to plurianticanonical sections on $Y$ and restricts to the coordinates $x,y,z_0,t_1$ of the ambient weighted projective space of $E$.
It is clear that
\[
(x = y = z_0 = t_1 = 0) \cap E = \emptyset
\]
since $\alpha \ne 0$.
This shows that $\pi_Y$ is everywhere defined.
We see $\pi_Y^* H = B$ and we compute $(H^3) = 1/24$ and 
\[
(B^3) = (A^3) - \frac{1}{3^3} (E^3) = \frac{1}{4} - \frac{1}{6} = \frac{1}{12}.
\]
From this we see that $\pi_Y$ is surjective and has degree $2$.
\end{proof}

By the same argument as in the proof of Proposition \ref{prop:deg20-birinv}, the above lemma implies the following.

\begin{Prop} \label{prop:deg4-birinv}
One of the following holds.
\begin{enumerate}
\item $\msp$ is not a maximal centre.
\item There is a birational involution $\sigma \colon X \ratmap X$ which is a Sarkisov link centred at $\msp$.
\end{enumerate}
\end{Prop}

\subsection{The $\frac{1}{4} (1,1,3)$ point and birational non-rigidity}

Let $\msp$ be the point of type $\frac{1}{4} (1,1,3)$.
We will show that the Kawamata blowup $\varphi \colon Y \to X$ leads to a Sarkisov link to a Mori fiber space which is not isomorphic to $X$.
The arguments are similar to those in Section \ref{sec:deg12nonbr} but more complicated.
Note that the $X$ has three points of type $\frac{1}{3} (1,1,2)$, denoted $\msq_1, \msq_2, \msq_3$.
We choose coordinates as in Lemma \ref{lem:coord1/4} for the $\frac{1}{3} (1,1,2)$ point $\msq_1$ and the $\frac{1}{4} (1,1,3)$ point $\msp$, so that $\msq_1 = \msp_{z_1}$ and $\msp = \msp_{t_1}$.

Recall that Lemma \ref{lem:coord1/4} is based on Condition \ref{cd:deg4-3} which we assume in this subsection.
In addition we assume the following condition which is satisfied for a general $X$.

\begin{Cond} \label{cd:deg4-4}
Under the choice of coordinates as in Lemma \ref{lem:coord1/4}, $y \in a_2$ and the set 
\[
(a_2 = b_3  = y + u e_1 + u^2 = 0) \subset \mbP (1_x,2_y,3_{z_0},1_u)
\]
consists of distinct $2$ points.
\end{Cond}

The Kawamata blowup $\varphi \colon Y \to X$ at $\msp$ is realized as the embedded weighted blowup with the initial weight $\wt (x,y,z_0,z_1,t_0,u) = \iniw = \frac{1}{4} (1,2,3,3,4,1)$ and we have an isomorphism
\[
E \cong (z_1 + u a_2 = t_0 + u b_3 = y + u e_1 + u^2 = 0) \subset \mbP,
\]
where $\mbP = \mbP (1_x,2_y,3_{z_0},3_{z_1},4_{t_0},1_u)$.
Let $X \ratmap \mbP (1,2,3,3,4)$ be the projection to $x,y,z_0,z_1,t_0$ and denote by $Z$ its image.
Let $\rho \colon Y \ratmap Z$ be the induced map.

\begin{Lem}
$\rho$ is a flopping contraction.
\end{Lem}

\begin{proof}
By Lemma \ref{lem:coord1/4}, it is easy to observe that the projection $X \ratmap \mbP (1,2,3,3,4)$ is defined outside $\msp$.
The sections $x,y,z_0,z_1,t_0$ lift to plurianticanonical sections on $Y$ and they restrict to $E$ the coordinates $x,y,z_0,z_1,t_0$ of $\mbP$.
We see
\[
(x = y = z_0 = z_1 = t_0 = 0) \cap E = \emptyset
\]
and this shows that $\rho$ is a morphism.
By the same argument as in the proof Lemma \ref{lem:deg12birrho}, we see that $\rho$ is birational and is the anticanonical model of $Y$.
The set $(a_2 = b_3 = 0) \cap E$ consits of two points by Condition \ref{cd:deg4-4} and it is mapped to the same point via $\rho$, which shows that $\rho$ is not an isomorphism.

It remains to show that $\rho$ is small.
Assume that $\rho$ is divisorial and let $G$ be the prime divisor on $Y$ contracted by $\rho$.
Since $(B^2 \cdot G) = 0$, we have $G \sim_{\mbQ} m (2 B - E)$ for some positive integer $m$.
By the same argument as in the proof of Lemma \ref{lem:deg12flop}, the proper transform $\tilde{H}$ of the divisor $H$ on $X$ defined by $z_1 b_3 - t_0 a_2 = 0$ satisfies $\tilde{H} \sim_{\mbQ} 6 B - E$.
By Lemma \ref{lem:divcontmob}, a component of $\tilde{H}$ which is $\mbQ$-linearly equivalent to $\lambda B - \mu E$ for some $\lambda,\mu$ with $\mu > 0$ is $G$.
It follows that $\tilde{H}$ contains $G$ as a component.
This in particular implies $m \le 2$.
We see that $\varphi_* G \sim_{\mbQ} 2 m A$ is cut out on $X$ by a polynomial of degree $2 m$ with $2 m = 2, 4$.
Hence $\varphi_*G$ contains the three singular points of type $\frac{1}{3} (1,1,2)$, and we conclude that $H$ contains the three singular points of type $\frac{1}{3} (1,1,2)$. 
But this is impossible since $H \sim_{\mbQ} 6 A$, which is defined by $z_1 b_3 - t_0 a_2 = 0$, contains at most $2$ singular points of type $\frac{1}{3} (1,1,2)$.
This is a contradiction and $\rho$ is a flipping contraction.
\end{proof}

Let $\varphi'_1 \colon Y'_1 \to X$ be the Kawamata blowup at the $\frac{1}{3} (1,1,2)$ point $\msq_1$ with exceptional divisor $E'_1$.
As is argued in the previous subsection, $\varphi'_1$ is realized as the embedded weighted blowup at $\msq_1 = \msp_{z_1}$ with weight $\wt (x,y,z_0,t_0,t_1,u) = \frac{1}{3} (1,2,3,1,4,2)$ and we have an isomorphism
\[
E'_1 \cong (t_1 + u a_2 + t_0 a_3 = u + t_0^2 + \gamma t_0 x = z_0 + u t_0 + \delta u x = 0) \subset \mbP',
\]
for some $\gamma,\delta \in \mbC$, where $\mbP' = \mbP (1_x,2_y,3_{z_0},1_{t_0},4_{t_1},2_u)$.

Let $\psi_1 \colon \hat{Y}_1 \to Y$ be the Kawamata blowup of $Y$ at the $\frac{1}{3} (1,1,2)$ point $\varphi^{-1} (\msq_1)$.
We have a natural birational morphism $\psi'_1 \colon \hat{Y}_1 \to Y'_1$ which is the Kawamata blowup of the $\frac{1}{4} (1,1,3)$ point ${\varphi'}_1^{-1} (\msp)$.
We see that the proper transforms $\hat{E}_1$ and $\hat{E}'_1$ of $E$ and $E'_1$ are the exceptional divisors of $\psi'_1$ and $\psi_1$, respectively. 
We denote by $\pi_1 \colon X \ratmap \mbP (1,2,3)$ the projection to $x,y,z_0$ and by $\eta_1 \colon \hat{Y} \ratmap \mbP (1,2,3)$ the induced rational map.
We set $B = -K_Y$ and $\hat{B} = - K_{\hat{Y}_1}$.

\begin{Lem} \label{lem:dge1/4ellipfib}
$\eta_1$ is a morphism which is an elliptic fibration.
Moreover, $\hat{E}_1$ and $\hat{E}'_1$ are respectively $2$-section and $3$-section of $\eta_1$.
\end{Lem}

\begin{proof}
We first show that $\pi_1 \colon X \ratmap \mbP (1,2,3)$ is defined outside the set $\{\msq_1, \msp\} = \{\msp_{z_1}, \msp_{t_1}\}$.
The indeterminacy locus of $\pi_1$ is the set $\Xi := (x = y = z_0 = 0) \cap X$.
We have 
\[
F_1 (0,0,0,z_1,t_0,t_1,u) = t_1 z_1, \ 
F_2 (0,0,0,z_1,t_0,t_1,u) = t_1 t_0,
\]
so that $\Xi = (x = y = z_0 = t_1 = 0) \cup (x = y = z_0 = z_1 = t_1 = 0)$.
By looking at the other polynomials $F_3,F_4,F_5$, it is easy to check that the former and the latter sets are $\{\msp_{z_1}\}$ and $\{\msp_{t_1}\}$, respectively, so that $\Xi = \{\msp_{z_1},\msp_{t_1}\}$.
It is straightforward to see $(x = y = z_0 = 0) \cap E = (x = y = z_0 = 0) \cap E' = \emptyset$, which shows that $\eta_1$ is a morphism.
Since $x,y,z_0$ lift to sections of $\hat{B}, 2 \hat{B}, 3 \hat{B}$, respectively, $\eta_1$ is the anticanonical morphism of $\hat{Y}_1$, that is, it is an elliptic fibration.

For $\lambda,\mu \in \mbC$, we set $S_{\lambda} = (y - \lambda x^2 = 0) \cap X$ and $T_{\mu} = (z_0 - \mu x^3 = 0) \cap X$.
We see that $\tilde{S}_{\lambda} \cap \tilde{T}_{\mu}$, where $\tilde{S}_{\lambda}, \tilde{T}_{\lambda}$ are the proper transforms of $S_{\lambda}, T_{\mu}$ via $\varphi$, is the fiber of $\pi_1 \circ \varphi \colon Y \ratmap \mbP (1,2,3)$ over the point $(1\!:\!\lambda\!:\!\mu)$ and we compute 
\[
(\tilde{S}_{\lambda} \cdot \tilde{T}_{\lambda} \cdot E) = (\tilde{S}_{\lambda}|_E \cdot \tilde{T}_{\mu}|_E)_E = \frac{2 \cdot 3 \cdot 3 \cdot 4 \cdot 2}{1 \cdot 2 \cdot 3 \cdot 3 \cdot 4 \cdot 1} = 2.
\]
Thus $\hat{E}_1$ is $2$-section of $\eta_1$.
Similarly, $S'_{\lambda} \cap T'_{\mu}$, where $S'_{\lambda}, T'_{\mu}$ are the proper transforms of $S_{\lambda}, T_{\mu}$ via $\varphi'_1$, is a fiber of $\pi_1 \circ \varphi'_1 \colon Y'_1 \ratmap \mbP (1,2,3)$ over the point $(1\!:\!\lambda\!:\!\mu)$ and we compute
\[
(S'_{\lambda} \cdot T'_{\lambda} \cdot E'_1) = (S'_{\lambda}|_{E'_1} \cdot T'_{\mu}|_{E'_1})_{E'_1} = \frac{2 \cdot 3 \cdot 4 \cdot 2 \cdot 3}{1 \cdot 2 \cdot 3 \cdot 1 \cdot 4 \cdot 2} = 3.
\]
This shows that $\hat{E}'_1$ is a $3$-section of $\eta_1$. 
\end{proof}

The above arguments hold true for $\msq_i$, $i = 2,3$, instead of $\msq_1$ (by re-choosing coordinates as in Lemma \ref{lem:coord1/4} for $\msq_i$ and $\msp$) and we obtain the following diagram for $i = 1,2,3$. 
\[
\xymatrix{
& \ar[ld]_{\psi'_i} \hat{Y}_i \ar[rd]^{\psi_i} \ar@/^2.3pc/[rrrddd]^{\eta_i} & & & \\
Y'_i \ar[rd]_{\varphi'_i} & & \ar[ld]^{\varphi} Y \ar[rd]_{\rho} & & \\
& X \ar@{-->}[rrrd]_{\pi_i \hspace{0.6cm}} \ar@{-->}[rr]& & Z \ar@{-->}[rd] & \\
& & & & \mbP (1,2,3)}
\]
where $\varphi'_i \colon Y'_i \to X$, $\psi_i \colon \hat{Y}_i \to Y$, $\psi'_i \colon \hat{Y}_i \to Y'_i$ are the Kawamata blowups at $\msq_i \in X$, $\varphi^{-1} (\msq_i) \in Y$, ${\varphi'}_i^{-1} (\msp) \in Y'_i$, respectively, $\eta_i \colon \hat{Y}_i \to \mbP (1,2,3)$ is the elliptic fibration induced by the natural projection $\pi_i \colon X \ratmap \mbP (1,2,3)$.
Let $E'_i$ be the $\varphi'_i$-exceptional divisor and $\hat{E}_i, \hat{E}'_i$ be the proper transform of $E$ and $E'_i$ via $\psi_i, \psi'_i$, respectively.
By Lemma \ref{lem:dge1/4ellipfib}, $\hat{E}_i$ and $\hat{E}'_i$ are $2$-section and $3$-section of $\eta_i$, respectively.

\begin{Thm} \label{thm:deg4nonrig}
The Sarkisov link $\sigma$ starting with the Kawamata blowup of $X$ at the $\frac{1}{4} (1,1,3)$ point is a link to a Mori fiber space which is not isomorphic to $X$.
In particular, $X$ is not birationally rigid.
\end{Thm}

\begin{proof}
Assume that the link $\sigma$ is an birational automorphism.
Then, by the same argument as in the proof of Theorem \ref{thm:deg12nonrig}, we obtain the flop $\tau$ of $\rho \colon X \to Z$, which is a birational automorphism sitting in a diagram
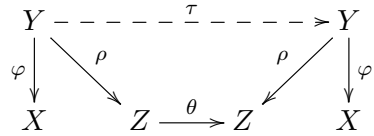
\[
\xymatrix{
Y \ar[d]_{\varphi} \ar[rd]^{\rho} \ar@{-->}[rrr]^{\tau}& & & \ar[ld]_{\rho} Y \ar[d]^{\varphi} \\
X & Z \ar[r]^{\theta} & Z & X} 
\]
where $\theta$ is an automorphism.
Note that $Y$ has four points of type $\frac{1}{3} (1,1,2)$, that is, $\varphi^{-1} (\msq_i)$ for $i = 1,2,3$ and the point $\bar{\msq}$ on the exceptional divisor $E$.
By the same argument as in the proof of Theorem \ref{thm:deg12nonrig}, the curves contracted by $\rho$ does not pass through $\varphi^{-1} (\msq_i)$ for $i = 1,2,3$, hence $\rho$ is an isomorphism around $\varphi^{-1} (\msq_i)$.
We set $\bar{\msq}_i = \rho (\varphi^{-1} (\msq_i)) \in Z$ which is of type $\frac{1}{3} (1,1,2)$, and $\bar{\msq} = \rho (\msq)$.
Since $\theta$ is an automorphism, it maps $\frac{1}{3} (1,1,2)$ point to a $\frac{1}{3} (1,1,2)$, and the set of $\frac{1}{3} (1,1,2)$ points on $Z$ is contained in $\{\bar{\msq}_1,\dots,\bar{\msq}_3,\bar{\msq}\}$.
By renumbering, we may assume that $\theta (\bar{\msq}_1) \ne \bar{\msq}$.
Set $\theta (\bar{\msq}_1) = \bar{\msq}_j$, $j \in \{1,2,3\}$.

For $i = 1,j$, let $\hat{\rho}_i \colon \hat{Y}_i \to \hat{Z}_i$ be the morphism induced by $N_i = -K_{\hat{Y}_i} + \varepsilon \hat{E}'_i$ for a sufficiently small $\varepsilon > 0$.
By the same argument as in the proof of Theorem \ref{thm:deg12nonrig}, $\hat{\rho}_i$ is a flopping contraction, and $\hat{Z}_i$ is obtained as the Kawamata blowup of $Z$ at $\bar{\msq}_i$.
Now, since $\theta (\bar{\msq}_1) = \bar{\msq}_j$, the automorphism $\theta \colon Z \to Z$ induces an isomorphism $\hat{\theta} \colon \hat{Z}_1 \to \hat{Z}_j$, and we have the following diagram
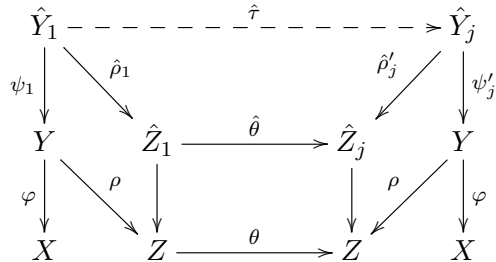
\[
\xymatrix{
\hat{Y}_1 \ar[d]_{\psi_1} \ar[rd]^{\hat{\rho}_1} \ar@{-->}[rrrr]^{\hat{\tau}} & & & & \ar[ld]_{\hat{\rho}'_j} \hat{Y}_j \ar[d]^{\psi'_j} \\
Y \ar[d]_{\varphi} \ar[rd]^{\rho} & \hat{Z}_1 \ar[d] \ar[rr]^{\hat{\theta}} & & \hat{Z}_j \ar[d] & \ar[ld]_{\rho} Y \ar[d]^{\varphi} \\
X & Z \ar[rr]^{\theta} & & Z & X}
\]
where $\hat{\tau} \colon \hat{Y}_1 \ratmap \hat{Y}_j$ is the map induced by $\tau \colon Y \ratmap Y$.
By construction, $\hat{\tau}_* \hat{E}'_1 = \hat{E}'_j$.
Hence $\hat{\tau}$ is an isomorphism in codimension one, that is, it is a flop.
By considering the anticanonical models of $\hat{Y}_1$ and $\hat{Y}_j$, we obtain an automorphism of $\mbP (1,2,3)$ sitting in the commutative diagram
\[
\xymatrix{
\hat{Y}_1 \ar[d]_{\eta_1} \ar@{-->}[r]^{\hat{\tau}} & \hat{Y}_j \ar[d]^{\eta_j} \\
\mbP (1,2,3) \ar[r]^{\cong} & \mbP (1,2,3)}
\]
and $\hat{\tau}$ induces an isomorphism between generic fibers of the elliptic fibrations $\eta_1, \eta_j$.

We set $\hat{B}_1 = - K_{\hat{Y}_1}$ and $\hat{B}_j = - K_{\hat{Y}_j}$.
Then $\hat{\tau}_* \hat{B}_1 = \hat{B}_j$ and $\hat{\tau}_* \hat{E}'_1 = \hat{E}'_j$.
We can write $\hat{\tau}_* \hat{E}_1 = \alpha \hat{B}_j - \beta \hat{E}_j + \gamma \hat{E}'_j$ for some integers $\alpha,\beta,\gamma$.
Since $\hat{\tau}_*$ induces an isomorphism between the divisor class groups, we have
\[
\begin{pmatrix}
1 & 0 & 0 \\
\alpha & - \beta & \gamma \\
0 & 0 & 1
\end{pmatrix} \in \operatorname{GL}_3 (\mbZ),
\]
which implies $\beta = 1$.
Since $\hat{\tau}_*\hat{E}_1$, $\hat{E}_j$ and $\hat{E}'_j$ are $2$-, $2$- and $3$-sections of $\eta_j$, respectively, the computation of intersection numbers of $\hat{\tau}_* \hat{E}_1 = \alpha \hat{B}_j - \hat{E}_j + \gamma \hat{E}'_j$ and a general fiber $C$ of $\eta_j$ gives $\gamma = 4/3$.
This is a contradiction since $\gamma$ is an integer and the proof is completed.
\end{proof}

\section{The table} \label{sec:table}

\newlength{\myheight}
\setlength{\myheight}{0.55cm}

We summarize the result of this paper in the following table.
The first column indicates the number and the type of singular points of $X$.
The second column indicates the existence of Sarkisov link centred at the corresponding point: If the second column is blank, then the corresponding point is not a maximal centre, and the mark ``Q.I." and ``$\exists$ Link" indicate that there is a Sarkisov link centred at the point which is a quadratic involution and a link to a Mori fiber space not isomorphic to $X$, respectively.
The third column indicates the generality condition required to prove the result indicated in the second column. \\

\begin{center}
\begin{flushleft}
$X_{16,17,18,19,20} \subset \mbP (1_x,5_y,6_z,7_t,8_u,9_v,10_w)$; $(A^3)
 = 1/42$.
\end{flushleft} \nopagebreak
\begin{tabular}{|p{60pt}|p{50pt}|>{\centering\arraybackslash}p{40pt}||p{60pt}|p{60pt}|>{\centering\arraybackslash}p{40pt}|}
\hline
\parbox[c][\myheight][c]{0cm}{} $\frac{1}{2} (1,1,1)$ & & no & $\frac{1}{3} (1,1,2)$ & & no \\
\hline
\parbox[c][\myheight][c]{0cm}{} $\frac{1}{5} (1,1,4)$ & & \ref{cd:deg42-5} & $\frac{1}{5} (1,2,3)$ & & no \\
\hline
\parbox[c][\myheight][c]{0cm}{} $\frac{1}{7} (1,1,6)$ & & no & & & \\
\hline
\end{tabular}
\end{center}

\begin{center}
\begin{flushleft}
$X_{14,15,16,17,18} \subset \mbP (1_x,5_y^2,6_z,7_t,8_u,9_v)$; 
$(A^3) = 1/30$.
\end{flushleft} \nopagebreak
\begin{tabular}{|p{60pt}|p{50pt}|>{\centering\arraybackslash}p{40pt}||p{60pt}|p{60pt}|>{\centering\arraybackslash}p{40pt}|}
\hline
\parbox[c][\myheight][c]{0cm}{} $\frac{1}{5} (1,1,4)$ & & \ref{cd:deg30-5} & $2 \! \times \! \frac{1}{5} (1,2,3)$ & & no \\
\hline
\parbox[c][\myheight][c]{0cm}{} $\frac{1}{6} (1,1,5)$ & & no & & &\\
\hline
\end{tabular}
\end{center}

\begin{center}
\begin{flushleft}
$X_{12,13,14,15,16} \subset \mbP (1_x,4_y,5_z^2,6_t,7_u,8_v)$; 
$(A^3) = 1/20$.
\end{flushleft} \nopagebreak
\begin{tabular}{|p{60pt}|p{50pt}|>{\centering\arraybackslash}p{40pt}||p{60pt}|p{60pt}|>{\centering\arraybackslash}p{40pt}|}
\hline
\parbox[c][\myheight][c]{0cm}{} $\frac{1}{2} (1,1,1)$ &  &  no & $\frac{1}{4} (1,1,3)$ & & \ref{cd:deg20-4} \\
\hline
\parbox[c][\myheight][c]{0cm}{} $2 \! \times \! \frac{1}{5} (1,1,4)$ & & no & $\frac{1}{5} (1,2,3)$ & Q.I. & \\
\hline
\end{tabular}
\end{center}

\begin{center}
\begin{flushleft}
$X_{10,11,12,13,14} \subset \mbP (1_x,3_y,4_z,5_t^2,6_u,7_v)$; 
$(A^3) = 1/12$.
\end{flushleft} \nopagebreak
\begin{tabular}{|p{60pt}|p{50pt}|>{\centering\arraybackslash}p{40pt}||p{60pt}|p{60pt}|>{\centering\arraybackslash}p{40pt}|}
\hline
\parbox[c][\myheight][c]{0cm}{} $2 \!\times \! \frac{1}{3} (1,1,2)$ & & \ref{cd:deg12-3} & $\frac{1}{4} (1,1,3)$ & & \ref{cd:deg12-4} \\
\hline
\parbox[c][\myheight][c]{0cm}{} $\frac{1}{5} (1,1,4)$ & Q.I. & no & $\frac{1}{5} (1,2,3)$ & $\exists$ Link & \ref{cd:deg12-5link} \\
\hline
\end{tabular}
\end{center}

\begin{center}
\begin{flushleft}
$X_{7,8,8,9,10} \subset \mbP (1_x,2_y,3_z^2,4_t^2,5_u)$; 
$(A^3) = 1/4$.
\end{flushleft} \nopagebreak
\begin{tabular}{|p{60pt}|p{50pt}|>{\centering\arraybackslash}p{40pt}||p{60pt}|p{60pt}|>{\centering\arraybackslash}p{40pt}|}
\hline
\parbox[c][\myheight][c]{0cm}{} $3 \!\times \! \frac{1}{2} (1,1,1)$ & & \ref{cd:deg4-2} & $3 \! \times \! \frac{1}{3} (1,1,2)$ & Q.I. & \ref{cd:deg4-3} \\
\hline
\parbox[c][\myheight][c]{0cm}{} $\frac{1}{4} (1,1,3)$ & $\exists$ Link & \ref{cd:deg4-3}, \ref{cd:deg4-4} & & & \\
\hline
\end{tabular}
\end{center}

%\bibliographystyle{amsplain}
%\bibliography{biblio}

\end{document}